\documentclass[11pt,reqno]{amsart}
\usepackage[dvipsnames]{xcolor}
\usepackage{setspace}
\usepackage[hang,flushmargin,symbol*]{footmisc}
\usepackage{amsmath}
\usepackage{amsthm}
\usepackage{amssymb}
\usepackage{mathtools}
\usepackage{enumitem}
\usepackage{calc}
\usepackage{graphicx}
\usepackage{caption}
\usepackage[labelformat=simple,labelfont={}]{subcaption}
\usepackage{tikz}
\usepackage{tikz-cd}
\usetikzlibrary{decorations.markings}
\usetikzlibrary{arrows,shapes,positioning}
\usepackage{url}
\usepackage{array}
\usepackage{graphicx}
\usepackage{color}
\usepackage{mathrsfs}

\usepackage{verbatim}

\usepackage{subfiles}

\usepackage[colorinlistoftodos, loadshadowlibrary]{todonotes}

\usepackage{dashbox}

\theoremstyle{definition}
\newtheorem{theorem}{Theorem}[section]
\newtheorem{lemma}[theorem]{Lemma}
\newtheorem{grule}[theorem]{Rule}

\newtheorem{corollary}[theorem]{Corollary}
\newtheorem{proposition}[theorem]{Proposition}
\newtheorem{warning}[theorem]{Warning}
\newtheorem{conjecture}[theorem]{Conjecture}
\newtheorem{definition}[theorem]{Definition}
\newtheorem{example}[theorem]{Example}
\newtheorem{remark}[theorem]{Remark}
\newtheorem{question}[theorem]{Question}

\newtheorem{assumption}[theorem]{Assumption}

\newtheorem{mainthm}{Theorem}

\usepackage{marginnote}

\usepackage{amsfonts}
\usepackage{amsmath}
\usepackage{amssymb}
\usepackage{mathrsfs}
\usepackage{tikz}
\usepackage{tikz-cd}
\usepackage[colorlinks=true]{hyperref}
\usepackage{ mathdots }
\usepackage{dutchcal}
\usepackage{verbatim}
\usepackage{amsmath,amsthm}
\usepackage{extarrows}

\usepackage{xparse}

\newcommand{\Glog}{\mathbb{G}_{\rm log}}
\newcommand{\Pic}{{\rm Pic}}

\usepackage{multirow}

\usepackage[normalem]{ulem}

\renewcommand{\tilde}[1]{\ensuremath{\widetilde{#1}}}

\newcommand{\pt}{{\rm pt}}

\newcommand{\msout}[1]{\text{\sout{\ensuremath{#1}}}}

\newcommand{\ZZ}{\mathbb{Z}}
\newcommand{\CC}{\mathbb{C}}
\newcommand{\NN}{{\mathbb{N}}}

\newcommand{\OO}{\mathcal{O}}

\newcommand{\Aff}{{\mathbb{A}}}
\newcommand{\PP}{\mathbb{P}}
\newcommand{\GG}{\mathbb{G}}

\newcommand{\Spec}{{\text{Spec}\:}}

\newcommand{\Hom}{\text{Hom}}

\newcommand{\Ext}{\text{Ext}}
\newcommand{\RHom}{\text{RHom}}

\newcommand{\dgcat}{\text{dgcat}}

\newcommand{\gp}[1]{#1^{gp}}

\newcommand{\Cl}[1]{{C_{#1}^{\rm log}}}

\newcommand{\Tl}[1]{{T^{\rm log}_{#1}}}
\newcommand{\lkah}[1]{\Omega^{1,{\rm log}}_{#1}}

\newcommand{\Log}{{\mathcal{L}}}

\newcommand{\lpb}{{\arrow[dr, phantom, very near start, "\ulcorner \ell"]}}
\newcommand{\lpbstrict}{{\arrow[dr, phantom, very near start, "\ulcorner \msout{\ell}"]}}
\newcommand{\pb}{{\arrow[dr, phantom, very near start, "\ulcorner"]}}

\newcommand{\bigslant}[2]{{\raisebox{.4em}{$#1$}\left/\raisebox{-.4em}{$#2$}\right.}}

\newcommand{\action}{\:\rotatebox[origin=c]{-90}{$\circlearrowright$}\:}

\newcommand{\Sym}{{\rm Sym}}

\renewcommand{\bar}[1]{\overline{#1}}

\newcommand{\HHl}[1]{{\rm HH}^\ell_{#1}}
\newcommand{\HH}[1]{{\rm HH}_{#1}}
\newcommand{\cHHl}[1]{{\rm HH}^{\ell #1}}

\newcommand{\lccx}[1]{\mathbb{L}^{\rm log}_{#1}}
\newcommand{\ccx}[1]{\mathbb{L}_{#1}}
\newcommand{\Map}{\mathbf{Map}}
\newcommand{\alg}{\mathbf{alg}}
\newcommand{\SSets}{\mathbf{SSets}}
\newcommand{\dSt}{\mathbf{dSt}}
\newcommand{\RSHom}{\mathcal{RHom}}
\newcommand{\D}{\mathbf{D}}

\newcommand{\longsimeq}{\overset{\sim}{\longrightarrow}}

\usepackage{stmaryrd}

\definecolor{sebgreen1}{rgb}{0.019,0.317,0.149}
\definecolor{sebgreen2}{rgb}{0.784,0.952,0.780}

\newcommand{\Leo}[2][inline]{\todo[linecolor=purple,backgroundcolor=purple!25,bordercolor=purple,#1,shadow,author=Leo]{#2}} 

\newcommand{\Marci}[2][inline]{\todo[linecolor=purple,backgroundcolor=green!25,bordercolor=purple,#1,shadow,author=Marci]{#2}} 

\newcommand{\Francesca}[2][inline]{\todo[linecolor=blue,backgroundcolor=blue!25,bordercolor=blue,#1,shadow,author=Francesca]{#2}} 

\newcommand{\af}[1]{\Theta_{#1}}

\newcommand{\scr}[1]{{\ensuremath{\mathscr{#1}}}}

\newcommand{\bra}[1]{{\left[{#1}\right]}}

\title[Logarithmic Hochschild co/homology]{Logarithmic Hochschild co/homology via formality of derived intersections}

\author{Márton Hablicsek}
\author{Leo Herr}
\author{Francesca Leonardi}
\date{\today}

\begin{document}

\maketitle

\begin{abstract}

We define log Hochschild co/homology for log schemes 
that behaves well for simple normal crossing pairs $(X,D)$ or toroidal singularities. 

We prove a Hochschild-Kostant-Rosenberg isomorphism for log smooth schemes, as well as an equivariant version for log orbifolds. We define cyclic homology and compute it in simple cases. We show that log Hochschild co/homology is invariant under log alterations.

Our main technical result in log geometry shows the tropicalization (Artin fan) of a product of log schemes $X \times Y$ is usually the product of the tropicalizations of $X$ and $Y$. This and the machinery of \emph{formality} of derived intersections facilitate a geometric approach to log Hochschild.

\end{abstract}



\section{Introduction}

\subsection{Overview}

The purpose of this paper is to provide a new, geometric framework to define invariants of log smooth logarithmic schemes/stacks using their combinatorics (their Artin fan). Log smooth log schemes include
\begin{enumerate}
    \item\label{eitem:logexamples1} A normal crossings or s.n.c. divisor $D$ on a smooth scheme $X$.
    \item A toric variety or toric singularity $X$ with its toric divisor $D$. 
    \item\label{eitem:logexamples4} A family of nodal curves $C \to S$ is log smooth (\emph{relatively} over $S$). 
\end{enumerate}
We define logarithmic variants of Hochschild co/homology and cyclic homology. This paper will:
\begin{enumerate}
    \item Define logarithmic Hochschild co/homology theory using the Artin fan of a log scheme/stack and show that the definition agrees with Olsson's definition \cite{olssondraftloghochschild}.
    \item Prove a Hochschild-Kostant-Rosenberg isomorphism \cite{kostant2009differential} in the case of log smooth schemes and log orbifolds.
    \item Establish \emph{formality} for log Hochschild co/homology and consequently invariance under log alterations.
    \item Identify an $S^1$-action on log Hochschild homology to define log cyclic homology using the Artin fan of a log scheme/stack.
\end{enumerate}



\subsection{Why Hochschild co/homology?}\label{sec:hochschild}

Hochschild co/homology is a fundamental algebrogeometric co/homology theory first introduced for associative algebras \cite{hochschild1945cohomology}, then generalized to
schemes and stacks \cite{swan1996hochschild} and to abelian and dg-categories \cite{lowen-vandenbergh, keller2021hochschild} which plays an important role in deformation theory, deformation quantization, and mirror symmetry.

For smooth and proper schemes $X$ over $\CC$, Hochschild homology encodes the vertical parts of the Hodge diamond via the celebrated Hochschild-Kostant-Rosenberg (HKR) isomorphism. If $X$ is Calabi-Yau, the Hochschild cohomology also encodes the horizontal parts of the Hodge diamond, i.e., the Betti numbers of $X$. 



Because Hochschild co/homology is defined for abelian and dg categories, it can define classical invariants for non-commutative schemes. 
Hochschild co/homology is invariant under Morita equivalence for abelian categories (and derived equivalence). In this way, Hochschild homology can be thought of as the non-commutative counterpart of Hodge cohomology. Cyclic homology, the non-commutative version of de Rham cohomology, can similarly be defined through Hochschild homology. 
This defines non-commutative Chern classes for instance, via the Dennis trace map from algebraic $K$-theory to cyclic homology.
\\



This paper takes a
geometric approach to Hochschild and cyclic co/homology for logarithmic schemes and stacks. This is a step towards non-commutative log geometry that circumvents the problem of \emph{log derived categories}. To that end, we include log orbifolds.


\subsection{What are log schemes?}

A log structure on a scheme encodes a boundary divisor, where it may have
mild singularities. 

Normal crossings pairs $(X^\circ, D)$ locally admit a map to $\Aff^n$
\[\exists_{loc} : X \to \Aff^n\]
such that the divisor $D$ is pulled back from the coordinate hyperplanes $V(x_1\cdots x_n) \subseteq \Aff^n$. Log schemes $X = (X^\circ, D)$ similarly admit a local map
\[\exists_{loc} : X \to \Aff_P = \Spec \ZZ[P]\]
to a toric variety such that $D$ is the pullback of the toric divisors on $\Aff_P$. These local maps are called ``charts'', and they encode how components of the toric divisor $D$ intersect.

If $V$ is a toric variety such as $\Aff_P$, the toric divisors and their intersections are visible already in the fan $\Sigma_V$. This fan consists of cones in a lattice $\ZZ^n$. The cones of $\Sigma_V$ are in bijection with torus orbits $V/T$, where $T$ is the dense torus. The finite set $V/T$ is a finite topological space with a continuous map
\[V \to V/T\]
that contains all the information of $V$. 

So normal-crossings pairs, resp.\ log schemes $X = (X^\circ, D)$ admit local maps to finite topological spaces
\[\exists_{loc} : X \to \Aff^n/\GG_m^n, \qquad \text{resp. } \exists_{loc} : X \to V/T\]
coming from the toric varieties $\Aff^n, V$. See Figure \ref{fig:A2conecomplex}. But there still may be no globally defined map $X \to V/T$. 

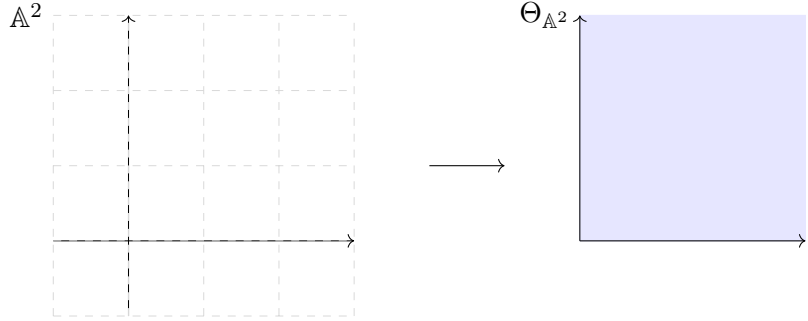
\begin{figure}
    \centering
    \begin{tikzpicture}
    \draw[->] (-1, 0) to (3, 0);
    \draw[->] (0, -1) to (0, 3);
    \draw[help lines, color=gray!30, dashed] (-1,-1) grid (3, 3);
    \node[left] at (-1, 3) {$\Aff^2$};
    \draw[->] (4, 1) to (5, 1);
    \begin{scope}[shift = {(6, 0)}]
    \filldraw[color=blue!10] (0, 0) rectangle (3, 3);
    \draw[->] (0, 0) to (3, 0);
    \draw[->] (0, 0) to (0, 3);
    \node[left] at (0, 3) {$\af{\Aff^2}$};
    \end{scope}
    \end{tikzpicture}
    \caption{The affine plane $\Aff^2$ and its fan $\Sigma_{\Aff^2} = \NN^2$. There are four cones in $\Sigma_{\Aff^2}$: the origin, the $x-$ and $y-$axes, and the first quadrant. They correspond respectively to the torus orbits in $\Aff^2$ of the torus itself, the $x-$ and $y-$axes, and the origin. The rays of the cone complex $\af{\Aff^2}$ pull back to the components of the toric divisor $V(xy) \subseteq \Aff^2$. The space $\af{\Aff^2}$ is the cone over the intersection complex of the divisor $V(xy) \subseteq \Aff^2$, which is a line segment. }
    \label{fig:A2conecomplex}
\end{figure}

For normal-crossings pairs, we can take the cone $\af{X}$ over the \emph{intersection complex} to get a global map
\[X \to \af{X}.\]
The intersection complex has a point for each component $D_i \subseteq D$, an edge for each nonempty intersection $D_i \cap D_j$, a 2-simplex for each nonempty intersection $D_i \cap D_j \cap D_k$, etc. Locally, each intersection looks like the topological space $\Aff^n/\GG_m^n$, so $\af{X}$ is the union of several copies of $\Aff^n/\GG_m^n$.

There is likewise a finite topological space $\af{X}$ for a log scheme $X$ with a well defined map
\[X \to \af{X}\]
such that $D \subseteq X^\circ$ is pulled back from the toric divisors on $\af{X}$. The space $\af{X}$ has a cover by $V/T$, ranging over various toric varieties $V$. This idea goes back to \cite{kempfknudsenmumfordsaintdonatkkmstoroidalembeddings}. We use the \emph{Artin fan} $\af{X}$ \cite{wisebounded}, an algebraic stack that is better behaved. See \cite{rendimentodeicontiwisepandharipandeherrmymolcho} for comparison and \cite{functorialtropicalizationulirsch} for alternatives.

One can \emph{define} log schemes as pairs of a scheme $X$ with a map
\[X \to \af{X}\]
to a certain finite topological space (algebraic stack) built as the union of torus orbits $V/T$ of toric varieties $V$. All schemes have a trivial log structure $X \to T/T = \pt$ where $D = \varnothing$.

\subsubsection{The log diagonal}
The main technical problem with Artin fans is that they are not functorial. Given a map $X \to Y$ of log schemes, there need not be a commutative square
\[
\begin{tikzcd}
X \ar[r] \ar[d] \ar[dr, phantom, "\times"]       &Y \ar[d]      \\
\af{X} \ar[r]      &\af{Y}.
\end{tikzcd}    
\]
Such squares exist locally on $X$ and $Y$, but not globally. 

Log Hochschild homology revolves around the diagonal map $\Delta : X \to X \times X$. Our main technical contribution is that, for the diagonal $\Delta$, there is indeed a commutative square
\[
\begin{tikzcd}
X \ar[r] \ar[d] \ar[dr, phantom, "\circ"]       &X \times X \ar[d]         \\
\af{X} \ar[r]      &\af{X \times X}.
\end{tikzcd}
\]
This results from our characterization of the Artin fan of a product of log smooth log schemes.

\begin{mainthm}[{Theorem \ref{thm:AFdiagonal}}]\label{mainthm:artinfan}

Suppose $X, Y$ are log smooth and quasicompact. Then the Artin fan of the product $X \times Y$ coincides with the product of the Artin fans
\[\af{X \times Y} \longsimeq \af{X} \times \af{Y}. \]
    
\end{mainthm}

We prove a relative version over a base $S$. Along the way, we show the geometric fibers of the tropicalization map $X \to \af{X/S}$ are connected in greater generality than was previously known.

\subsection{Hochschild co/homology of log schemes}

Let $X$ be a quasi-compact, separated, Noetherian scheme. One can define its ordinary Hochschild co/homology (Section \S \ref{sec:hochschild}) using either:
\begin{itemize}
    \item the bar resolution, i.e. the derived self-intersection $X \to X\times X$, or 
    \item the abelian or dg category of coherent sheaves on $X$.
\end{itemize}
The two definitions coincide \cite{lowen-vandenbergh, keller2021hochschild}. 

For log schemes $X$, neither works. 
\textbf{Problem A:} The diagonal map $X\to X\times X$ is not strict, i.e. it does not respect log structures. 
This makes it difficult to understand the logarithmic derived self-intersection. 
\textbf{Problem B:} There is no widely accepted definition of an abelian (or dg) category of coherent sheaves on a log scheme \cite{mehta1980moduli, yokogawa1995infinitesimal, talpo2018infinite, vaintrob2017categorical}. See Remark \ref{rem:nologcat}. 

We define log Hochschild co/homology using the combinatorics of Artin fans to circumvent these problems. By Theorem \ref{mainthm:artinfan}, there is a commutative, non-Cartesian square
\begin{equation}\label{eqn:mainAFdiagonal}
\begin{tikzcd}
X \ar[r] \ar[d]       &X \times X \ar[d]         \\
\af{X} \ar[r]      &\af{X \times X}.
\end{tikzcd}   
\end{equation}

Write $B = B_X$ for the pullback, which factors the diagonal
\[X \to B \to X \times X.\]
The map $i : X \to B$ is strict and $B \to X \times X$ is log étale. 

\begin{remark}

In the semistable case, $B$ recovers Kato-Saito's ``log diagonal'' \cite{katosaitologdiagonal} as in Remark \ref{rmk:BvsKatoSaitologdiagonal}. They used it to define a log intersection product that coincides with the ones defined in \cite{mythesislogprodfmla}, \cite{logquantumkproductchouherrlee}. 

\end{remark}

To overcome \textbf{Problem A}, we define log Hochschild co/homology via the derived self-intersection of the strict map $i:X\to B$. The log \'etale map $B\to X\times X$ does not affect the log Hochschild homology. We show this derived self-intersection is \emph{formal} in the sense of \cite{arinkin-caldararu}.

To avoid \textbf{Problem B}, we define the Hochschild homology $\HHl{X}$ of a log scheme/stack as an endofunctor 
\[i^*i_*:\D(X)\to \D(X)\]
of a dg enhancement $\D(X)$ of the derived category of coherent sheaves on the underlying scheme $X^\circ$. Here the functors $i_*$ and $i^*$ are derived.

\begin{question}

Does our notion of log Hochschild co/homology coincide with the ones defined using any of the ``log derived categories'' from Remark \ref{rem:nologcat}? 
    
\end{question}

For technical reasons, the answer should be no. But we suspect one can recover categorical log Hochschild from our notion. Our notion may coincide with one defined using the \emph{rhizomic topology} \cite{rhizomic}.



\subsection{The geometric approach}

In this paper, we provide a new, geometric approach to define logarithmic Hochschild homology for schemes and stacks providing a complementary approach to \cite{HKRlogHH}. This approach has many advantages that we highlight below.

\subsubsection{Formality}

Our version of the logarithmic Hochschild homology is defined via a derived self-intersection $i:X\to B$. We call this derived self-intersection the \textit{logarithmic derived loop space}. Using the factorization of the diagonal $X\to B\to X\times X$, we show that the logarithmic derived loop space is formal in the sense of \cite{arinkin2012self}.


\begin{mainthm}[{Theorems \ref{thm:formality}, \ref{thm:loghkr}}]\label{mainthm:loghkr}

Let $X$ be a quasicompact, weakly log separated (Definition \ref{defn:logsep}), log smooth log scheme. 

\begin{itemize}

\item The endofunctor $i^*i_*$ is formal: there exists an isomorphism of dg endofunctors $\D(X)\to \D(X)$
\[i^*i_*(-) \longsimeq (-)\otimes \Sym(\Omega^{1, \log}_X[1]).\]
Here, $\Sym$ is the derived symmetric algebra. 

\item The log Hochschild homology of $\OO_X$ can be computed in terms of the log cotangent bundle:
\begin{equation}\label{eqn:mainloghkr}
R^n\Gamma(X,\HHl{X}(\OO_X))=\oplus_{q-p=n}H^p(X,\Omega^{q,\log}_X).    
\end{equation}

\end{itemize}
  
\end{mainthm}

The theorem above is a natural generalization of the Hochschild-Kostant-Rosenberg isomorphism (HKR) \cite{kostant2009differential} on ordinary Hochschild homology providing a new proof of the logarithmic HKR theorem \cite{HKRlogHH}. However, the formality of the logarithmic derived loop space gives a much stronger result strengthening the result of \cite{HKRlogHH}, namely, the first part of Theorem \ref{mainthm:loghkr} tells us that \emph{any} perfect complex enjoys an HKR isomorphism similar to the second part of Theorem \ref{mainthm:loghkr}. 

The chain complexes in \eqref{eqn:mainloghkr} each have natural algebra structures. The algebra structure on the log Hochschild homology $R\Gamma(X,\HHl{X}(\OO_X))$ comes from the structure sheaf of the derived self-intersection of $X\to B$, and $\oplus_{q-p=\star}H^p(X,\Omega^{q,\log}_X)$ has the wedge product. Using the formality of the logarithmic derived loop space, the second part of Theorem \ref{mainthm:loghkr} can be promoted to a natural isomorphism of graded $k$-algebras.

Dual versions of Theorem \ref{mainthm:loghkr} hold for log Hochschild \emph{co}homology. However, the algebra structure on the log Hochschild cohomology is much more involved. We conjecture Kontsevich's Duflo isomorphism \cite{kontsevich2003deformation, calaque2010hochschild} for log Hochschild cohomology in Section \S \ref{ss:duflo}.

\subsubsection{Invariance of log Hochschild homology under log alterations}

If $D = X \setminus U$ is an s.n.c. compactification of an open variety, blowing up the strata of $D$ results in a ``log blowup'' of $X$. These yield different compactifications of $U$. 

More generally, we can take roots along components of $D$ to obtain ``log alterations.'' Our log Hochschild co/homology is ``logarithmically independent'' of the choice of compactification $U \subseteq X$. 

\begin{mainthm}[{Theorem \ref{thm:logHHlogaltn}}]\label{mainthm:logHHalterations}

Let $\pi : X \to Y$ be log étale, for example a log alteration. Under assumptions on $Y$, the log Hochschild co/homology of a sheaf/complex $F$ on $Y$ pulls back to that of $\pi^*F$ on $X$:
\[\pi^* \HHl Y (F) = \HHl X (\pi^* F), \qquad \pi^* \cHHl Y (F) = \cHHl X (\pi^* F).\]

\end{mainthm}

\begin{remark}
    A simple but important application is the case of a log blow-up. Let $\pi:(X,D_X)\to (Y, D_Y)$ be a log blow-up of snc varieties, $X$ and $Y$. Then, the theorem above implies that $\pi^*$ induces an isomorphism of log Hochschild homology groups of $(X,D_X)$ and $(Y,D_Y)$
\[R\Gamma(X,\HHl{X}(\OO_X)\simeq R\Gamma(Y,\HHl{Y}(\OO_Y)).\]
This isomorphism provides evidence that the log derived categories of log blow-ups should be equivalent.
\end{remark}

\subsubsection{Cyclic homology}

The log diagonal admits an equivalent description
\[B = X\times X \times_{\af{X\times X}}\af{X} \simeq X\times_{\af{X}}X.\]
Therefore, the logarithmic derived loop space (i.e, the derived self-intersection of $X\to B$) is equivalent to that of the diagonal map $X\to X\times_{\af{X}}X$ over $\af{X}$. This describes the logarithmic derived loop space as derived mapping space $\Map_{\dSt/\af{X}}(S^1,X)$.

The log Hochschild homology of $\OO_X$ is thereby equipped with an $S^1$-action. This allows us to define variants of \emph{cyclic homology} in the case of log schemes. The relation between log Hochschild homology and ordinary Hochschild homology coming from formality also provides a simple way to compute cyclic homology for a pair $(X, D)$ where $X$ is a smooth scheme and $D$ a simple normal crossing divisor. Defining variants of cyclic homology in the logarithmic setting has been a difficulty in the literature, however, with our geometric approach this is straightforward work.

\subsubsection{Log orbifolds}

We conclude our paper by considering log (quotient) orbifolds $[X/G]$. Let $G$ be a finite group acting on a log scheme $X$ and trivially on its Artin fan $\af{X}$. Using the formality of derived intersections \cite{grivaux2020derived, arinkin2019formality}, we show the log derived loop space of an orbifold is formal over its log inertia stack. This allows us to compute log Hochschild homology with an equivariant HKR theorem.

\begin{mainthm}[Corollary \ref{cor:logorbifolddecom}]\label{mainthm:orbifold}

For a firm action $G \action X$ (Assumption \ref{ass:actionstrata}), we have isomorphisms of $k$-vector spaces
\[R^n\Gamma(\HHl{[X/G]}(\OO_{[X/G]}))=\left(\bigoplus_{g\in G} R^n\Gamma(\HHl{X^g_{\log}}(\OO_{X^g_{\log}}))\right)^G=\]
\[=\left(\bigoplus_{g\in G} \bigoplus_{q-p=n}H^p(X^g_{\log}, \Omega^{q,\log}_{X^g_{\log}})\right)^G.\]

\end{mainthm}

\subsection{Other versions of log Hochschild co/homology}

There are two distinct notions of log Hochschild co/homology, in analogy with the two versions of the cotangent complex due to Gabber and Olsson \cite{logcotangent}. The difference comes from two separate ways of deriving a log structure. 

Olsson created a moduli space $\Log$ of log structures \cite{logstacks}. Any log scheme $X$ has a unique map $X \to \Log$ parameterizing its log structure. Olsson's log cotangent complex is the relative cotangent complex of the map $X \to \Log$ \cite{logcotangent}:
\[\lccx{X} = \ccx{X/\Log}.\]
See Section \ref{s:logbasics} for more. 

Gabber instead replaces the monoid $M_X$ by a simplicial monoid \cite[\S 8]{logcotangent}. The Gabber log cotangent complex comes from resolving the ring $\OO_X$ and the monoid $M_X$ simultaneously. Remarkably, the two log cotangent complexes coincide up to degree $-3$ \cite[Theorem 8.32]{logcotangent}.

A simplicial monoid may interpreted as an $E_\infty$ space $\tilde M$, the ring as an $E_\infty$ ring spectrum $\tilde R$. But the usual definition of log structure does not make sense for a map $\epsilon : \tilde M \to \tilde R$. Neither do the notions of ``fine'' or ``saturated.''

For spectral analogues of ``f.s. log structures'', J. Rognes suggested the notion of \emph{repletion} \cite{rognesoriginalloghh}. See also \cite{rognes2014}, \cite{rognes2015}. Log Hochschild may be defined using the replete diagonal by \cite[Definition 5.3]{HKRlogHH}. See loc. cit. for more on Rognes' log Hochschild.

Repletion is the right notion for Rognes' spectral log structures. However, these do not have the same combinatorial nature or modular interpretation. For ordinary log schemes, such as smooth schemes with s.n.c. divisors, toroidal singularities, etc., the Olsson approach taken here has advantages. For example, log schemes have a \textit{tropicalization} known as an Artin fan that encodes the combinatorics of how its strata meet.

We define a Hochschild co/homology for log schemes in the spirit of Olsson's log cotangent complex and tropicalization, using his stack $\Log$. Olsson defined one by a concrete chain complex using descent theory for sheaves on Artin stacks in forthcoming work \cite{olssondraftloghochschild}.

The log Hochschild co/homology defined in this article coincides with that of M. Olsson by Proposition \ref{prop:olssonvsus}. We rediscovered it by working only with the tropicalizations of log schemes. We suggest the moniker ``Olsson log Hochschild co/homology'' for the one here and in \cite{olssondraftloghochschild}, as opposed to Rognes' version.

\begin{question}[{\cite[\S 1.14]{HKRlogHH}}]

The two log cotangent complexes agree up to degree $-3$. Can the same techniques equate the two versions of log Hochschild co/homology in low degrees? 
    
\end{question}

Using our main theorem, D.~Park sketched an argument for us answering this question in the affirmative when the two log cotangent complexes coincide.


\subsection*{Conventions}

We assume the reader is familiar with log geometry at the introductory level \cite{kato1989logarithmic}, \cite{ogusloggeom}, \cite{rendimentodeicontiwisepandharipandeherrmymolcho}, \cite{loggeomintrohandbook}. We use more involved log techniques from \cite{logstacks} in Section \ref{s:AFdiagonal}. 

Our functors are always derived where appropriate. For example, $i_*$ means the derived functor $Ri_*$: we drop the $R$ or $L$. In particular, ``$\Sym$'' refers to the derived dg symmetric algebra functor. It is the quotient of the tensor algebra by the relation $x \otimes y - y \otimes x$ for elements of degree 1. The induced relation for higher-degree terms is $ x \otimes  y + (-1)^{\deg  x \deg  y} y \otimes x$. For a complex $E[-1]$ concentrated in degree 1 in characteristic 0, the symmetric algebra is isomorphic to the complex $\bigoplus_n \bigwedge^n E$ with the 0-differentials.

Write $\Aff_P \coloneqq \Spec \ZZ[P]$ and $\af{P}$ for the quotient stack
\[\af{P} \coloneqq \bra{\Spec \ZZ[P]/\Spec \ZZ[\gp P]} = \bra{\Aff_P/\Aff_{\gp{P}}}\]
of the affine toric variety $\Aff_P$ by its dense torus. It admits a log structure by fppf descent from the canonical log structure on $\Aff_P$. Abbreviate $\af{} = \af{\NN}$. 

The f.s. fiber product differs from the fiber product of schemes. We write $\ulcorner \ell$ or $X \times_S^\ell Y$ for the f.s. fiber product and $\ulcorner\msout{\ell}$ or $X \times_S^{\msout{\ell}} Y$ for when they happen to coincide. This is just a mnemonic for the reader, and we also state clearly in which category a fiber product is to be taken. 

We say a map $X \to Y$ is a \emph{log alteration} if, strict étale locally in $Y$, it is pulled back from a map $\scr B \to \scr C$ of Artin fans which is of DM type, proper, and birational. These are locally pulled back from a morphism of toric varieties consisting of subdividing the fan and rescaling the lattice. Log alterations are further
\begin{itemize}
    \item \emph{Log modifications} if they are representable. This entails a subdivision of the toric fans without rescaling the lattice. 
    \item \emph{Log blowups} if they are representable and projective,
    \item \emph{Root stacks} if they are integral, which means only rescaling is allowed and no subdivisions. 
    
\end{itemize}

By \textit{tropicalization} of a log scheme, we refer to its Artin fan. See \cite{functorialtropicalizationulirsch} for other notions and some comparison. 

From Section \ref{sec:loghoch} on, we work over a sufficiently large field $k$, not necessarily algebraically closed. Its characteristic must be zero or sufficiently large compared to the dimension $\dim X$. 


We work with log algebraic stacks $(X,M_X)$ with f.s. log structure $M_X$ and Artin fan $X \to \af{X}$ \cite[Proposition 3.2.1]{wisebounded}. We assume every algebraic stack is locally of finite type, hence locally noetherian. We write $\Log$ for M. Olsson's stack of f.s. log structures $\scr T or$ \cite{logstacks}.

\subsection*{Acknowledgments}

Jonathan Wise inspired the techniques of \S \ref{s:AFdiagonal} in another context. Martin Olsson shared an early draft of \cite{olssondraftloghochschild} in which he first defines log Hochschild co/homology. We thank Pieter Belmans, David Holmes, Y.P. Lee and Cris Negron for helpful conversations. Finally, HHL dedicate HHL to LCI.

\section{Log Basics}\label{s:logbasics}

We ask experts to skip this section. Details and introductions may be found in \cite{kato1989logarithmic}, \cite{ogusloggeom}, \cite{rendimentodeicontiwisepandharipandeherrmymolcho}, \cite[\S 1 Logarithmic geometry and moduli]{loggeomintrohandbook}. 

Let $\Log$ be the stack of log structures. A map from an ordinary scheme to $\Log$ is the same as a log structure on $T$:
\[\Log(T) \coloneqq \left\{M_T \text{ log structure on } T\right\}.\]
There is a relative version as well. Let $Y$ be a log scheme. Define $\Log Y$ as log structures together with a morphism to $Y$:
\[\Log Y (T) \coloneqq \left\{M_T \text{ log structure on } T  + (T, M_T) \to (Y, M_Y)\right\}.\]
The stacks $\Log$ and $\Log Y$ have universal log structures, which pull back to the $M_T$ along the corresponding map. That is, the map $T \to \Log Y$ parameterizing $M_T$ is \emph{strict}. 

\begin{grule}\label{rule:logblank}
Say a map $X \to Y$ of log schemes is log flat, log smooth, log étale, log $\cdots$ if the corresponding map $X \to \Log Y$ is flat, smooth, étale, $\cdots$. 
\end{grule}

The log cotangent complex and normal cone are defined the same way:
\[\lccx{X/Y} = \ccx{X/\Log Y}, \qquad \Cl{X/Y} = C_{X/\Log Y}.\]
Any map of log schemes $X\to Y$ factories through $\Log Y$ by definition. The important thing about the factorization
\[X \overset{i}{\longrightarrow} \Log Y \overset{g}{\longrightarrow} Y\]
is that $i$ is strict and $g$ is log étale. Given another factorization
\[X \to B \to Y\]
into strict and log étale maps, log constructions for $X/Y$ are similarly equated with the ordinary ones for $X/B$. 

We consider schemes $X$ as log schemes, endowed with their natural initial log structure $M_X = \OO_X^*$. We use $X^\circ$ to denote the underlying scheme of a log scheme. 

Every toric variety has a natural log structure from its toric divisors. For example, $\Aff^1$ has the origin, $\Aff^2$ the two axes, and $\Aff^n$ the $n$ hyperplanes. All our log schemes are f.s., so they locally admit strict maps to affine toric varieties 
\[\exists_{loc} f : X \to \Aff_P, \qquad f^*M_{\Aff_P} = M_X.\]

One can obtain a global chart if one takes Artin fans instead of affine toric varieties, gluing together all the local charts:
\[X \to \af{X}. \]
We assume our log schemes and log algebraic stacks $X$ are locally of finite type, so they all admit Artin fans \cite{wisebounded}. The Cantor set with constant log structure does not admit an Artin fan in the usual sense.

A morphism of log schemes $f : X \to Y$ is a morphism of schemes together with a map on log structures $M_Y|_X \to M_X$. A technical problem in log geometry is whether such a map $f$ is compatible with a morphism of Artin fans
\[
\begin{tikzcd}
X \ar[r] \ar[d]       &Y \ar[d]      \\
\af{X} \ar[r, dashed, "\exists?"]      &\af{Y}.
\end{tikzcd}
\]
Locally on $X$ and $Y$, this is always possible. If $f$ is integral or exact, it is possible \cite{rendimentodeicontiwisepandharipandeherrmymolcho}. However, there are examples of $f : X \to Y$ where no such map $\af{X} \dashrightarrow \af{Y}$ exists \cite[5.4.1]{skeletonsfansabramchenmarcusulrischwise}.

Another central technical problem is that the fiber product in f.s. log schemes differs from the fiber product in schemes. There is always a map
\begin{equation}\label{eqn:fspb}
\begin{tikzcd}
X \times^\ell_Z Y \ar[dr] \ar[r, "j"] \ar[rr, bend left=15]       &X \times_Z^{sch} Y \ar[r] \ar[d] \pb         &X \ar[d]      \\
        &Y \ar[r]      &Z
\end{tikzcd}  
\end{equation}
from the f.s. fiber product $X \times^\ell_Z Y$ to the fiber product in schemes $X \times^{sch}_Z Y$. The map is saturation, which is a finite surjection onto a closed immersion. We use the notation $\ulcorner, \times$ for ordinary pullbacks and $\ulcorner \ell, \times^\ell$ for f.s. pullbacks. We may write $\ulcorner \msout{\ell}$ when they coincide.

Here is an example where the map $j$ is a closed immersion:

\begin{example}

Let $X = Bl_{\vec 0} \Aff^2$ be the blowup of $\Aff^2$ at the origin. Give $\Aff^2$ the toric divisorial log structure from the $x-$ and $y-$axes and $X$ that from the strict transforms of the axes \emph{and} the exceptional divisor $E$.  

The ``log blowup'' $X \to \Aff^2$ is a log étale proper monomorphism. Consider the strict pullback
\[
\begin{tikzcd}
E \ar[r] \ar[d] \lpbstrict       &X \ar[d]      \\
\Vec{0} \ar[r]       &\Aff^2.
\end{tikzcd}
\]
Here, the origin $\Vec{0}$ has rank-two log structure. The morphism $E \to \Vec{0}$ is logarithmically a monomorphism, even though it very much is not on underlying schemes! The f.s. and schematic pullbacks are therefore
\[E \times_{\Vec{0}}^\ell E = E, \qquad E \times^{sch}_{\Vec{0}} E = \PP^1 \times \PP^1.\]
The map between them is the diagonal. 

\end{example}

And now for an example where the map $j$ is a finite surjection:

\begin{example}

Consider the map 
\[\bra{r} : \Aff^1 \to \Aff^1 \qquad t \mapsto t^r.\]
What is the f.s. pullback of two copies of this map?
\[
\begin{tikzcd}
? \ar[r] \ar[d] \lpb       &\Aff^1 \ar[d, "\bra{r}"]         \\
\Aff^1 \ar[r, "\bra{r}", swap]       &\Aff^1.
\end{tikzcd}
\]

The ordinary pullback is the scheme
\[X = \Spec \ZZ[t, u]/(t^r = u^r).\]
Over $\CC$, this is the union of $r$ lines $t = \omega^i u$, where $\omega$ is a primitive $r$th root of unity. They meet at the origin. This non-normal scheme is a degree $r$ ramified cover along each projection to $\Aff^1$. Its log structure is generated by elements $T, U$ mapping to $t, u$ modulo the relation $r T = r U$. This is not saturated. 

The saturation of $X$ is the normalization, namely $r$ disjoint lines
\[X^{sat} = \Aff^1 \sqcup \cdots \sqcup \Aff^1.\]

\end{example}

Saturation is not always the normalization of $X$, but rather that of its Artin fan $\af{X}$. Making a scheme integral gives a closed immersion cut out by monomial equations.

If $X \to S$ is a morphism, there may also be a relative Artin fan $\af{X/S}$. It is the initial factorization 
\[X \to \scr B \to \Log S\]
of the map $X \to \Log S$ where $\scr B \to \Log S$ is étale and representable. It is a \emph{family of cone stacks} over $S$ \cite{rendimentodeicontiwisepandharipandeherrmymolcho}, a space locally constructed from pullbacks
\[
\begin{tikzcd}
S \times_{\af{Q}} \af{P} \ar[r] \ar[d] \lpbstrict        &\af{P} \ar[d]         \\
S \ar[r]       &\af{Q}.
\end{tikzcd}    
\]
Relative Artin fans exist for log flat morphisms of finite presentation with ``log reduced geometric fibers'' \S \ref{ss:loggeomfibers}. 

\begin{warning}

This differs from the Artin fan $\scr A_{X/Y}$ \textit{of} a morphism $X \to Y$ \cite[Proposition 3.3.2]{wisebounded}. Their notion is the relative Artin fan $\scr A_{X/Y} = \af{X/\af{Y}}$ of the map $X \to \af{Y}$. We do not use this special case. 
    
\end{warning}

\begin{example}\label{ex:AFlsmpi0}
    
If $X \to S$ is log smooth and quasicompact, then $\af{X/S}$ coincides with the space of connected components
\[\af{X/S} = \pi_0(X/\Log S)\]
of \cite[Construction 6.8]{laumonmoretbaillystacks}, \cite{romagnypi0}. The map $X \to \af{X/S}$ is surjective, with geometrically connected fibers. 

\end{example}

We provide a more general criterion for when the map $X \to \af{X/S}$ has geometrically connected fibers in Proposition \ref{prop:AFconnfibers}.

\begin{remark}

If a map $X \to \af{}$ is flat, the pullback $D \coloneqq B \GG_m \times_{\af{}} X$ of the closed point of $\af{}$ is an effective Cartier divisor
\[
\begin{tikzcd}
D \ar[r] \ar[d] \pb       &B\GG_m \ar[d]         \\
X \ar[r, "\Psi"]      &\af{}.
\end{tikzcd}    
\]
Likewise every effective Cartier divisor $D \subseteq X$ fits in a unique such pullback square. 

Suppose we are over a field $k$ and $X$ is reduced. Then effective Cartier divisors always come from flat maps $\Psi$ by miracle flatness \cite[Proposition III.9.7]{hartshorne}. Otherwise we can have non-flat maps such as
\[\ZZ[t] \to \ZZ; \qquad t \mapsto 2\]
cutting out a Cartier divisor.

\end{remark}

\begin{example}

Let $D \subseteq X$ be an effective Cartier divisor and suppose both $X, D$ are smooth and geometrically connected. The map $\Psi : X \to \af{}$ parameterizing $D \subseteq X$ is the Artin fan of $X$ with divisorial log structure from $D$. We prove this in Example \ref{ex:AFsmdivisor}.

\end{example}

\section{Artin fan of the diagonal}\label{s:AFdiagonal}

The goal of this section is to prove our main technical result. 

\begin{theorem}\label{thm:AFdiagonal}

Consider a pair of log schemes or log algebraic stacks $X, Y$ over a base $S$ that admit relative Artin fans over $S$. Suppose the maps $X, Y \to S$ are quasicompact, log flat, and have log reduced geometric fibers (Definition \ref{defn:logreducedfibers}). 

Then the Artin fan of the fiber product $X \times_S^\ell Y$ is the fiber product of the Artin fans
\[\af{X \times_S^\ell Y/S} \longsimeq \af{X/S} \times_S^\ell \af{Y/S}. \]
This is an isomorphism of factorizations of $X \times_S^\ell Y \to \Log S$.

\end{theorem}

For the proof, it is enough that the fiber product $X \times_S^\ell Y \to \af{X \times_S^\ell Y/S}$ and one of $X \to \af{X/S}, Y \to \af{Y/S},$ have connected, nonempty geometric fibers. We show that holds under our assumptions in Proposition \ref{prop:AFconnfibers}. 

It suffices to take $X, Y \to S$ quasicompact and log smooth. We will only use the absolute special case for simplicity.

\begin{corollary}\label{cor:AFdiagonalabsolute}

If $X, Y$ are log smooth and quasicompact, the Artin fan of their product is the product of the Artin fans
\[\af{X \times Y} \longsimeq \af{X} \times \af{Y}. \]
    
\end{corollary}

The theorem concerns a fiber product. The product of two maps $X \to S, Y \to T$ is simpler. 

\begin{proposition}\label{prop:AFproduct}

Let $f : X \to S$, $g : Y \to T$ be a maps of log algebraic stacks admitting relative Artin fans. Suppose $f, g$ are quasicompact, log flat, and have log reduced geometric fibers. 

The product of the Artin fans is the Artin fan of the product:
\[\af{X \times Y /S \times T} \longsimeq \af{X/S} \times \af{Y/T}. \]
This is an isomorphism of factorizations of
\[X \times Y \to \Log S \times \Log T.\]
    
\end{proposition}

\subsection{Geometric fibers in log geometry}\label{ss:loggeomfibers}

We repeatedly reuse a lemma that lets us compare geometric fibers. We consider when a map $X \to Y$ of algebraic stacks has \emph{reduced} or \emph{connected} geometric fibers. 

The property ``is reduced'' is smooth-local, so a map $X \to S$ of algebraic stacks has \emph{reduced geometric fibers} if there is a square
\[
\begin{tikzcd}
U \ar[r] \ar[d]       &V \ar[d]      \\
X \ar[r]       &S
\end{tikzcd}
\]
with $V \to S$, $U \to X \times_S V$ smooth surjections from schemes such that $U \to V$ has reduced fibers. 

Connectedness of an algebraic stack $X$ means that of its underlying topological space $|X|$ of equivalence classes of points \cite[04XG]{sta}.

\begin{lemma}\label{lem:geomfibers}

Let $X \overset{f}{\to} Y \overset{g}{\to} Z$ be maps of algebraic stacks. Suppose $g : Y \to Z$ is étale and let $v \to Z$ be a geometric point. Ranging over lifts $v \dashrightarrow Y$ of the point to $Y$, to geometric fibers of $f$ form an étale cover of those of $g \circ f$
\begin{equation}\label{eqn:geomfibers}
\bigsqcup_{v \dashrightarrow Y} X \times_Y v \to X \times_Z v.   
\end{equation}
If $f$ has reduced geometric fibers, so does $g \circ f$. 

If $g$ is surjective and representable by algebraic spaces, we have a converse. The map \eqref{eqn:geomfibers} is the identity. If $g \circ f$ has reduced or connected geometric fibers, the same goes for $f$. 
    
\end{lemma}

\begin{proof}

Assume $Z = v$ is a geometric point. An étale map $Y \to Z$ is automatically of DM type. Choose an étale cover $Y' \to Y$ by a scheme $Y'$. The map $Y'\to Z$ is a disjoint union of copies of the points $\bigsqcup_I Z \to Z$ by \cite[02GL]{sta}. 

Any geometric point $v \to Y$ lifts to the cover $Y'$. Write $X_i$ for the preimage of the $i$th inclusion
\[
\begin{tikzcd}
X_i \ar[rr] \ar[d] \pb         &       &X \ar[d]      \\
Z \ar[r, "{\rm inc}_i"]       &Y'= \bigsqcup_I Z \ar[r]     &Y.
\end{tikzcd}
\]
The family $\{X_i \to X\}_I$ is an étale cover. 

If $g$ is representable by algebraic spaces and we assume $Z = v$, then $g$ is representable by schemes \cite[03KX]{sta}. Take $Y' = Y$ above. 


\end{proof}




We define when a map $X \to S$ has \emph{log reduced geometric fibers} according to the paradigm of Rule \ref{rule:logblank}. 

\begin{lemma}\label{lem:logreducedfibers}

For a morphism $X \to S$ of log algebraic stacks, the map $X \to \Log S$ has reduced geometric fibers if and only if $\Log X \to \Log S$ does. If a pair of maps $X, Y \to \Log S$ both have reduced geometric fibers, so does the fiber product $X\times_S^\ell Y \to \Log S$. 
    
\end{lemma}

\begin{proof}

By definition of reduced geometric fibers, we can replace $X$ by a smooth cover to assume it is a scheme. The second statement results from the first because of the strict pullback square
\[
\begin{tikzcd}
\Log (X \times_S^\ell Y) \ar[r] \ar[d] \lpbstrict      &\Log X \ar[d]         \\
\Log Y \ar[r]      &\Log S.
\end{tikzcd}    
\]

Consider the pullback square
\[
\begin{tikzcd}
\Log X \ar[r] \ar[d] \lpbstrict      &\Log^2 S \ar[r, "d_1"] \ar[d, "d_2"]       &\Log S         \\
X \ar[r]       &\Log S.
\end{tikzcd}    
\]
The stack $\Log^2 S$ \cite{logstacks} parameterizes chains of log structures 
\[M_S \to M_1 \to M_2.\]
The map $d_1$ forgets $M_1$ and $d_2$ forgets $M_2$. 


If $X \to \Log S$ has reduced geometric fibers, so does the pullback $\Log X \to \Log^2 S$. The composite $\Log X \to \Log^2 S \to \Log S$ does also, by Lemma \ref{lem:geomfibers}. The converse is immediate.  
    
\end{proof}

\begin{definition}\label{defn:logreducedfibers}

If $X \to \Log S$ has reduced geometric fibers, we say $X \to S$ has \emph{log reduced geometric fibers}. 
    
\end{definition}

Composites of morphisms with log reduced geometric fibers also have them.

\subsubsection{Weakly log separated}

We introduce a technical condition for later, that of being ``weakly log separated.'' 

\begin{definition}\label{defn:logsep}

Let $f : X \to S$ be a morphism of log algebraic stacks which admits an Artin fan $\af{X/S}$. Say $f$ is \emph{weakly log separated} over $S$ if the map $X \to \af{X/S}$ is separated, i.e., if the diagonal map
\[X \to X \times_{\af{X/S}} X\]
is proper. A log algebraic stack $X$ is \emph{weakly log separated} if the map $X \to \pt$ is. 
    
\end{definition}

The condition that $X \to \Log S$ is separated is stronger. There is a pullback diagram
\[
\begin{tikzcd}
X \ar[r]       &X \times_{\af{X/S}} X \ar[r] \ar[d] \lpbstrict      &X \times_{\Log S} X \ar[d]       \\
        &\af{X/S} \ar[r]       &\af{X/S} \times_{\Log S} \af{X/S}.
\end{tikzcd}
\]
Because $\af{X/S} \to \Log S$ is étale and representable by algebraic spaces, its diagonal is an open embedding. If the composite $X \to X \times_{\Log S} X$ is proper, so is the map $X \to X \times_{\af{X/S}} X$, which was claimed. 

All Artin fans $\scr B$ are weakly log separated, but the map $\scr B \to \Log$ need not be separated. Take the union $X = \Log \cup \Log$ along the open point $\pt \in \Log$. Then $X$ is its own Artin fan and weakly log separated even though the map $X \to \Log$ is not separated.

Toric varieties $V$ are weakly log separated, as the diagonal map
\[V \to V \times_{\af{V}} V\]
is the inclusion 
\[V \to V \times T, \qquad v \mapsto (v, 0).\]
If $X$ is a log scheme, it is weakly log separated if and only if $X \to X \times_{\af{X}} X$ is a closed immersion.

\begin{lemma}

Let $X \to Y$ be a strict, separated map of log algebraic stacks admitting Artin fans over $S$. If $Y$ is weakly log separated over $S$, so is $X$. 
    
\end{lemma}

\begin{proof}

Construct the diagram
\[
\begin{tikzcd}
X \ar[rr] \ar[dr, "\Delta_{X/Y}", swap]      &       &X \times_{\af{X/S}} X \ar[d, "t"]      \\
        &X \times_Y X\ar[r, "j'"] \ar[d] \lpbstrict       &X \times_{\af{Y/S}} X  \ar[d]            \\
        &Y \ar[r, "j"]      &Y \times_{\af{Y/S}} Y     
\end{tikzcd}
\]
with Cartesian square. By assumption, $j, j',$ and $\Delta_{X/Y}$ are proper. The map $\af{X/S} \to \af{Y/S}$ is étale and representable, so its diagonal is open and the same goes for its pullback $t$. Thus $X \to \af{X/S}$ is separated. 
    
\end{proof}

Being weakly log separated and separated are independent, as we now give examples of each condition without the other.

\begin{example}

Let $X = \Aff^1 \cup \Aff^1$ be the bug-eyed line at the origin, the union of two copies of $\Aff^1$ along the complements of their origins. Endow it with the divisorial log structure at both origins. 

The Artin fan $\af{X} = \af{} \cup \af{}$ is the union of two copies of $\af{} = \bra{\Aff^1/\GG_m}$ along their open point. The map $X \to \af{X}$ is the quotient by $\GG_m$, so it is separated. This toric bug-eyed line is weakly log separated.

The stack $\Log$ is weakly log separated but not separated.


\end{example}

For pairs $(X, D)$ that are smooth or normal crossings, the resulting log structure on $X$ is weakly log separated if and only if $X^\circ$ is separated by Proposition \ref{prop:logsepsnc}.

\subsection{Comparing Artin fans}

We establish a few lemmas to help show Artin fans are isomorphic.

\begin{lemma}\label{lem:bijectivegeomptsisom}

An étale morphism $g : Y \to Z$ representable by \emph{schemes} between locally noetherian algebraic stacks locally of finite type which is injective on geometric points is an open immersion. If it is bijective on geometric points, it is an isomorphism.

\end{lemma}

\begin{proof}

Localize to assume $Y, Z$ are schemes. Morphisms locally of finite type which are surjective on geometric points are surjective by \cite[0487]{sta}. It remains to prove the first statement. 

The map $Y \to Z$ is injective on geometric points if and only if $Y \to Y \times_Z Y$ is surjective on geometric points, hence surjective by \cite[0487]{sta} again. Then $Y \to Z$ is universally injective and radicial \cite[01S4]{sta}. A radicial étale map is an open immersion \cite[02LC]{sta}.

\end{proof}

\begin{lemma}\label{lem:artinfanisom}

Let $f : \scr B \to \scr C$ be a strict étale map (locally of finite type) representable by algebraic spaces that is a bijection on geometric points. Suppose $\scr B$ is a family of cone stacks over $S$. Then $f$ is an isomorphism. 
    
\end{lemma}

\begin{proof}

Localize to assume $\scr C = \af{Q} \times S$. We argue the map $f$ is proper by lifting maps from $\af{}$
\begin{equation}\label{eqn:properartinfanlifting}
\begin{tikzcd}
         &\scr B \ar[d]        \\
\af{} \ar[r] \ar[ur, dashed]         &\af{Q} \times S
\end{tikzcd}
\end{equation}
by \cite[Theorem 2.4.1]{birationalinvarianceabramovichwise}. Write $\scr B' \coloneqq \scr B \times_{\af{Q} \times S} \af{}$ for the pullback. 

Let $v$ be a geometric point with log structure $M_v = \OO^*_v \oplus \NN$ and strict map $v \to \af{}$. Lifts of \eqref{eqn:properartinfanlifting} are equivalent to strict lifts of the log geometric point $v$ and in turn to lifts of the underlying geometric point $v^\circ$
\[
\begin{tikzcd}
    &\scr B' \ar[d]        \\
v \ar[r] \ar[ur, dashed]      &\af{}
\end{tikzcd}  \qquad \equiv \qquad
\begin{tikzcd}
    &\scr B' \ar[d]        \\
v^\circ \ar[r] \ar[ur, dashed]         &\af{}.
\end{tikzcd}  
\]
But lifts of geometric points are unique by assumption. 

The map $f$ is proper étale, hence finite and representable by \emph{schemes}. We conclude by Lemma \ref{lem:bijectivegeomptsisom}.

\end{proof}

\begin{example}

The assumption that $\scr B \to \scr C$ is representable by algebraic spaces is necessary. Let $\scr B$ be any Artin fan and $\scr C$ \emph{the Artin fan} of $\scr B$. The two need not coincide, as in the famous example \cite[5.4.1]{skeletonsfansabramchenmarcusulrischwise}. 


\end{example}

\begin{proposition}\label{prop:AFisoconnfibers}

Let $\scr B \to \scr C$ be a strict étale map representable by algebraic spaces, with $\scr B$ a family of cone stacks over $S$. Suppose some map $X \to \scr B$ and the composite $X \to \scr B \to \scr C$ are surjective with geometrically connected fibers. Then $\scr B \longsimeq \scr C$ is an isomorphism.  
    
\end{proposition}

\begin{proof}

As in the proof of Lemma \ref{lem:geomfibers}, the fibers $\scr B \times_{\scr C} v$ over a geometric point $v$ are disjoint unions $\bigsqcup v$. If $\scr B \times_{\scr C} v \neq v$, the fibers of $X \to \scr C$ are not connected. 
The map $\scr B \to \scr C$ is bijective on geometric points, so Lemma \ref{lem:artinfanisom} asserts it is an isomorphism. 
    
\end{proof}

\begin{example}\label{ex:AFsmdivisor}

Let $D \subseteq X$ be an effective Cartier divisor, with both $X, D$ smooth and geometrically connected. We claim that the map $X \to \af{}$ is the Artin fan for the $D$-divisorial log structure on $X$. 

There is a sequence
\[X \to \af{X} \to \af{}.\]
The fibers of $X \to \af{}$ are geometrically connected -- $D$ is by assumption, and $X \setminus D$ is geometrically irreducible. So Proposition \ref{prop:AFisoconnfibers} equates $\af{X} \longsimeq \af{}$. 

Now let $D \subseteq X$ be an s.n.c. divisor and write $D_i$ for the $n$ components. All the strata $D_I = \bigcap_{i \in I} D_i$ are smooth, and suppose they are geometrically connected; this includes $D_\varnothing = X$. Suppose all the components meet in a single nonempty deepest stratum $D_I$. There is a map
\[X \to \af{}^n,\]
where each coordinate parameterizes a component $D_i$. Proposition \ref{prop:AFisoconnfibers} identifies this map with the Artin fan of $X$
\[\af{X} \simeq \af{}^n.\]

If the strata are not geometrically connected, this is false. Take a smooth, disconnected divisor $D = D_1 \sqcup D_2$ such as a pair of points
\[V(x(x-1)) \subseteq \Aff^1.\]
The Artin fan of $\Aff^1$ with the divisorial log structure is the union of two copies of $\af{}$ along their open points.
    
\end{example}

The map $\Log S \to S$ is quasiseparated in the sense of the stacks project but not in the sense of Olsson's article or Laumon-Moret-Bailly \cite[Remark 1.1]{logquantumkproductchouherrlee}. So if $X \to S$ is quasicompact, so is $X \to \Log S$ by \cite[050Y]{sta}.

\begin{proposition}\label{prop:AFconnfibers}

Let $X \to S$ be quasicompact, log flat, and have log reduced geometric fibers. Then the map $X \to \af{X/S}$ to the relative tropicalization is surjective and has connected geometric fibers. It coincides with the space of relative connected components $\pi_0(X/\Log S)$ \cite{romagnypi0}. 
    
\end{proposition}

If $X \to S$ is quasicompact and log smooth, the hypotheses are satisfied. Compare with \cite[Lemme 1.2.1]{ferrandgabberetaleenvelope}.

\begin{proof}

We assume all our stacks are locally noetherian and locally of finite type, so $X \to \Log S$ is flat and of finite presentation. All log flat maps $X \to S$ of finite presentation yield surjective maps $X \to \af{X/S}$, because the map to the Artin fan is open.

There is a quasicompact, étale factorization
\[X \to \pi_0(X/\Log S) \to \Log S\]
under our hypotheses by \cite[Theorem 2.5.2]{romagnypi0}. There is a unique factorization
\[X \to \af{X/S} \dashrightarrow \pi_0(X/\Log S) \to \Log S\]
by universal property. The dashed arrow is an étale cover representable by algebraic spaces.

Apply Lemma \ref{lem:geomfibers} to $X \to \af{X/S} \to \pi_0(X/\Log S)$, to see the geometric fibers of $X \to \af{X/S}$ are connected.  
%
%
Proposition \ref{prop:AFisoconnfibers} shows the dashed arrow is an isomorphism
\[\af{X/S} \longsimeq \pi_0(X/\Log S).\]

\end{proof}

\begin{example}[{\cite[Proposition 1.4.1]{ferrandgabberetaleenvelope}}]

Without log reduced fibers, there may not exist an relative tropicalization for $X \to S$. Suppose $X \to S$ is strict; then the relative tropicalization is the initial factorization of $X \to S$ through an étale $S$-algebraic space. If $X = \Spec \ZZ[i]$ and $S = \Spec \ZZ$, there is no such factorization. If $p \in \ZZ$ is an odd prime and $T_p = \Spec \ZZ \cup \Spec \ZZ$ is the bug-eyed line at $p$, there is a factorization
\[X \to T_p \to S.\]
But this is true for all odd primes $p$, so there can be no initial such factorization. 
    
\end{example}

\begin{corollary}\label{cor:universalAF}

Let $X \to S$ be quasicompact, log flat, and have log reduced geometric fibers. Consider a strict étale map $\scr B \to \af{X/S}$ from a family of cone stacks to the Artin fan of $X$ that is representable by algebraic spaces. Take the pullback
\[
\begin{tikzcd}
X' \ar[r] \ar[d] \lpbstrict      &X \ar[d]      \\
\scr B \ar[r]      &\af{X/S}.
\end{tikzcd}
\]  
Then $X' \to \scr B$ is the relative Artin fan of $X'$. 

If $S$ is atomic and $\scr B = \af{P} \times S$ is an $S$-Artin cone, then $X'$ is also atomic. 
    
\end{corollary}

\begin{proof}

Because $X \to \af{X/S}$ is quasicompact, flat, and has reduced fibers, the same goes for $X' \to \scr B$. There is a unique, strict étale surjection $\af{X'/S} \to \scr B$ representable by algebraic spaces. Lemma \ref{lem:geomfibers} applied to $X' \to \af{X'/S} \to \scr B$ shows $X'$ has log reduced geometric fibers. 

The map $X' \to \af{X'/S}$ is surjective and has connected geometric fibers by Proposition \ref{prop:AFconnfibers}. Then Proposition \ref{prop:AFisoconnfibers} equates $\af{X'/S} \longsimeq \scr B$. If $\scr B = \af{P} \times S$ and $S$ is atomic, so is $X'$. 
    
\end{proof}

\begin{proof}[Proof of Theorem \ref{thm:AFdiagonal}]

Form the f.s. pullback diagram
\[
\begin{tikzcd}
W \ar[r] \ar[d] \lpbstrict       &Y' \ar[r] \ar[d] \lpbstrict     &Y \ar[d]      \\
X' \ar[r] \ar[d] \lpbstrict      &\af{X/S} \times^\ell_S \af{Y/S} \ar[r] \ar[d] \lpb        &\af{Y/S} \ar[d]       \\
X \ar[r]       &\af{X/S} \ar[r]       &S,
\end{tikzcd}
\]
naming $W = X \times^\ell_S Y$. 

We assumed $X \to \af{X/S}$ was flat and quasicompact. It is of finite presentation by our assumptions, so \cite[06R7]{sta} asserts it is universally open. 

Consider the composite
\[W \overset{f}{\longrightarrow} Y' \overset{g}{\longrightarrow} \af{X/S} \times_S^\ell \af{Y/S}.\]
We want to show $g \circ f$ has geometrically connected fibers. We know $g$ does, and $f$ is open and has geometrically connected fibers. It follows that $g \circ f$ has geometrically connected fibers by \cite[0387]{sta}. 

The map $\af{X/S} \times_S^\ell \af{Y/S} \to S$ is étale and representable by algebraic spaces, engendering a map from the Artin fan
\[
\begin{tikzcd}
W \ar[r] \ar[d]       &\af{X/S} \times_S^\ell \af{Y/S} \ar[d]    \\
\af{W} \ar[r] \ar[ur, dashed, "\exists !"]      &\Log S.
\end{tikzcd}
\]
By Proposition \ref{prop:AFconnfibers}, the map $W \to \af{W}$ is surjective with geometrically connected fibers. Proposition \ref{prop:AFisoconnfibers} equates the Artin fans
\[\af{W} \longsimeq \af{X/S} \times_S^\ell \af{Y/S}.\]

\end{proof}

\begin{proof}[Proof of Proposition \ref{prop:AFproduct}]

Obtain maps
\[
\begin{tikzcd}
X \times Y \ar[r] \ar[d]      &\af{X/S} \times \af{Y/T} \ar[d]       \\
\af{X \times Y/S \times T} \ar[ur, dashed, "\exists !"] \ar[r]      &\Log S \times \Log T.
\end{tikzcd}        
\]
Argue the dashed arrow is an isomorphism using Proposition \ref{prop:AFconnfibers} and Proposition \ref{prop:AFisoconnfibers} as in the proof of Theorem \ref{thm:AFdiagonal}.

\end{proof}

\section{Log Hochschild Co/homology}\label{sec:loghoch}

In this section, we offer another construction of log Hochschild co/homology and equate it with M. Olsson's version \cite{olssondraftloghochschild}. 

Using the machinery of formality of derived self-intersections \cite{arinkin-caldararu}, we derive versions of the celebrated Hochschild-Konstant-Rosenberg (HKR) isomorphism \cite{kostant2009differential} and the Duflo isomorphism for log schemes \cite{kontsevich2003deformation, calaque2010hochschild}. 

We show our notion of log Hochschild co/homology is invariant under log alterations. We also define and compute log versions of cyclic homology.


Fix a finite type, quasicompact, quasiseparated log algebraic stack $X$. By Theorem \ref{thm:AFdiagonal}, we obtain a commutative diagram
\[\begin{tikzcd}
X \ar[r]\ar[d] & X\times X\ar[d]\\
\af{X}\ar[r]& \af{X\times X}.
\end{tikzcd}
\]
This diagram is almost never Cartesian. If it were, the diagonal $X \to X \times X$ would be log étale.

We consider the fibre product $B = B_X$
\begin{equation}\label{eq:diagram}
    \begin{tikzcd}
 X \ar[r, "i"]\ar[rd] & B\ar[r] \ar[d]& X\times X\ar[d]\\
& \af{X}\ar[r]& \af{X\times X}.
\end{tikzcd}
\end{equation}
This is both an f.s. and ordinary fiber product because the map $X \times X \to \af{X \times X}$ is strict. The map $i:X\to B$ is the \emph{log diagonal}.

\begin{proposition}\label{prop:logdiagonalprops}

The map $f : B \to X \times X$ is log étale and representable by algebraic spaces. It factors through an open embedding $B \subseteq X \times_\Log X$. 

The log diagonal $i : X \to B$ is strict, quasicompact, quasiseparated, and representable by algebraic spaces. 

\end{proposition}

\begin{proof}

The map $B \to X \times X$ is the pullback of a map between Artin fans, which is log étale by \cite[Lemma A.7]{abrammarcuswisecomparisontheoremsforgwisofsmoothpairs}.
To see that $i$ is strict it suffices to recall that the log structure of a fiber product is the coproduct of the log structures and to use the local explicit construction of the Artin fan of \cite{skeletonsfansabramchenmarcusulrischwise}. The diagonal $X \to X \times X$ is representable, so $i$ is once we show $B \to X \times X$ is. 

The diagonal of the algebraic stack $\af{X}$ is étale and representable by algebraic spaces; the same goes for its pullback $B \to X \times X$. 

There is a variant of \eqref{eq:diagram}
\[\begin{tikzcd}
X \ar[r] \ar[dr]      &B \ar[r] \ar[d] \lpbstrict       &X \times_\Log X \ar[d]         \\
        &\af{X} \ar[r]          &\af{X} \times_\Log \af{X}. 
\end{tikzcd}\]
Because $\af{X} \to \Log$ is strict, étale, and representable by algebraic spaces, its diagonal is open. 



\end{proof}

\begin{remark}\label{rmk:identifyingolssonversion}
Equivalently, $B = X \times_{\af{X}} X$ where the map $X\to \af{X}$ is the Artin fan map. The map $B \to \af{X}$ is strict, but it need not be \emph{the} Artin fan of $B$. 

If $X$ is log smooth, $B \to \af{X}$ is the Artin fan of $B$. Argue that the fibers of $B \to \af{X}$ are connected and conclude with Propositions \ref{prop:AFconnfibers}, \ref{prop:AFisoconnfibers}. 

The log diagonal $X \to B$ is proper if and only if $X$ is weakly log separated. If $X$ is a weakly log separated log scheme, $X \to B$ is a closed immersion. 

\end{remark}

\begin{example}\label{ex:toric}

If $X$ is a toric variety with dense torus $T$, the description $B = X \times_{\af{X}} X$ with $\af{X} = \bra{X/T}$ gives the identification
\[B \simeq X \times_{\bra{X/T}} X = X \times T. \]
The map $B \to X \times X$ is the product of the projection and action maps $X \times T \rightrightarrows X$. The factorization of $X \to X \times X$ through $B$ is the inclusion 
\[X \subseteq X \times T; \qquad x \mapsto (x, 0).\]
    
\end{example}

\begin{remark}\label{rmk:BvsKatoSaitologdiagonal}

The map $X \to B$ is essentially the ``log diagonal'' \cite[Definition 4.2.9]{katosaitologdiagonal}. If $X$ is atomic, then $(X \times_S X)\tilde{\phantom{a}}$ from \cite[\S 5.2]{katosaitologdiagonal} coincides with $B_X$ when it is defined. Before the advent of Artin fans, one worked locally with charts $[P] \coloneqq \af{P}$. Their definition also involves choices of charts and ``framing'', which can mean that their log diagonal
\[X \to ``X\times^{\log}_{S \times [P]} X"\]
is only exact \cite[Corollary 4.2.8]{katosaitologdiagonal}. 
    
\end{remark}

\begin{lemma}\label{lem:xtoblci}
     If $X$ is log smooth, then $B$ is also log smooth and $i : X \to B$ is a local complete intersection morphism. If $X$ is also weakly log separated, $X \to B$ is a regular embedding. 

\end{lemma}

\begin{proof}

The morphism $B \to \af{X}$ is strict and smooth, with log étale target. Its source is log smooth. 


Since $X \to B$ is representable by algebraic spaces, being a local complete intersection is smooth-local in the source and target\footnote{
Being a local complete intersection is not smooth-local for algebraic stacks, as they may have nonvanishing cotangent complex in positive degrees. 
}. Replace $X, B$ by schemes. 

Consider the distinguished triangle between the cotangent complexes coming from the sequence of maps $X\to B\to \af{X}$:
\[\ccx{B/\af{X}}|_X \to \lccx{X} \to \ccx{X/B} \xrightarrow{[+1]}.\]
Since $X\to \af{X}$ is log flat, $B$ is also equivalent to the derived fibre product $(X\times X)\times^h_{\af{X\times X}}\af{X}$, meaning that $\ccx{B/\af{X}}|_X=\lccx{X}\oplus \lccx{X}$.

By \cite{logcotangent} log smooth schemes have log cotangent complex isomorphic to the module of log K\"ahler differential forms. From the diagram above, we have that $\ccx{X/B}$ is concentrated in degree $-1$, meaning that $X\to B$ is a local complete intersection morphism. 

If $X$ is further weakly log separated, $i : X \to B$ is a closed and hence regular embedding.


\end{proof}


Now, we are ready to define log Hochschild co/homology.

\begin{definition}\label{def:loghoch}
Let $X$ be a quasicompact, quasiseparated, weakly log separated, finite type log algebraic stack. Let $\D(X)$ be a dg enhancement of the unbounded derived category of coherent sheaves on $X$. 

Define the \textbf{log Hochschild homology} $\HHl{X}$ as the endofunctor 
\[i^*i_*:\D(X)\to \D(X).\]

If $X$ is a log scheme, we define the \textbf{log Hochschild cohomology} $\cHHl{X}$ as the endofunctor 
\[i^!i_*:\D(X)\to \D(X).\]
Here $i^!$ denotes the right adjoint of $i_*$ that exists in our case by \cite{neeman1996grothendieck}.


\end{definition}

We omit the $\mathbf{R}$’s and $\mathbf{L}$’s in front of derived functors; all our functors are derived where appropriate.

\begin{remark}
    In the definition, we do not consider the log structure on $X$, since $X\to B$ is a strict morphism. 
\end{remark}

\begin{remark}\label{rem:nologcat}
    We defined Hochschild co/homology as an endofunctor because there are multiple different notions of coherent sheaves on logarithmic schemes in the literature, for instance, using parabolic bundles \cite{mehta1980moduli,yokogawa1995infinitesimal}, considering sheaves on root stacks \cite{talpo2018infinite}, or considering toroidal compactifications \cite{vaintrob2017categorical}.
    A subtlety is that $\OO_X$ and other vector bundles on $X$ need not be ``log coherent sheaves'' by these definitions, as $\OO_X$ does not satisfy descent for log alterations. 

    We conjecture that our notion recovers the categorical log Hochschild co/homology using any of these, but it should not be equal to it. We also expect the rhizomic topology \cite{rhizomic} produces a derived category whose Hochschild co/homology coincides with ours. 
    
\end{remark}

\begin{remark}

If $X$ is a log smooth scheme in Definition \ref{def:loghoch}, both $i^*i_*$ and $i^!i_*$ restrict to endofunctors of the derived category of perfect complexes on $X$. This is a consequence of Proposition \ref{prop:formalityloghochschild}.


\end{remark}

We only define log Hochschild for log algebraic stacks for comparison in the log orbifold case \S \ref{s:logorbifolds}. For the bulk of the paper, we assume $X$ is a log smooth log scheme.

\begin{assumption}\label{ass:logsmvar}

Assume $X$ is a finite type, weakly log separated, log smooth log scheme.

\end{assumption}

Note $X$ is automatically Noetherian, finite presentation, quasicompact, and quasiseparated.

Our log Hochschild co/homology recovers M. Olsson's construction.

\begin{proposition}\label{prop:olssonvsus}

Let $X$ satisfy Assumption \ref{ass:logsmvar}. The value of the log Hochschild homology functor $\HHl{X}$ on the structure sheaf recovers M. Olsson's construction of log Hochschild cohomology
\[\HHl{X}(\OO_X) \simeq HH(X/\Log)\]
as defined in \cite{olssondraftloghochschild}.

\end{proposition}

\begin{proof}

Olsson's explicit chain complex computes the derived self-intersection of the diagonal $X \to X \times_\Log X$ \cite{olssondraftloghochschild}. The map $u: B \to X \times_\Log X$ is open by Proposition \ref{prop:logdiagonalprops}, hence the natural transformation $u^* u_* \Rightarrow \rm{Id}$ is an equivalence. This makes the natural transformation $ i^*  u^* u_* i_* \Rightarrow i^* i_*$ an equivalence, that identifies Olsson's definition with the one given in this paper.
    
\end{proof}

For the most basic examples of normal crossings pairs $(X, D)$, we say when $X$ is weakly log separated. 

\begin{proposition}\label{prop:logsepsnc}

Let $(X, D)$ be a normal crossings pair with $X$ quasicompact. Assume all the strata $D_I = \bigcap_I D_i$ are smooth, including $D_\varnothing = X$. Endowing $X$ with the divisorial log structure, $X$ is weakly log separated if and only if $X^\circ$ is separated. 
    
\end{proposition}

This applies for example to smooth pairs $(X, D)$. 

\begin{proof}

The map $X \to \af{X}$ is smooth and quasicompact. Localize in $\af{X}$ to assume $\af{X} = \af{}^n$ and $X$ is atomic by Corollary \ref{cor:universalAF}. 

The map $X \to \af{}^n$ is weakly log separated if and only if its base change $L \to \Aff^n$ along the quotient $\Aff^n \to \af{}^n$ is separated. This $L$ is the direct sum of the line bundles $\OO(D_i)$ for the components of $D$. 

Argue $L$ is separated if and only if $X$ is. Use the fact that sections of a vector bundle are closed immersions.

\end{proof}

\subsection{Formality of Hochschild co/homology}\label{sec:formselfint}

In this subsection, we prove Theorem \ref{mainthm:loghkr}.

Given a closed embedding $i:X\to B$ of \textit{schemes}, one can consider the derived self-intersection $W \coloneqq X\times^h_B X$. The notion of \textit{formality} concerns the simplicity of this derived self-intersection \cite{arinkin-caldararu, grivaux2014hochschild, yu2016dolbeault}. 

From a geometric point of view, formality means that $W$ is the total space of a (shifted) vector bundle. From an algebraic point of view, formality assures that the structure sheaf of the derived self-intersection is as simple as possible. 

The structure sheaf of $W$ (over $X$) is given by the derived tensor product
\[\OO_W=\OO_X\otimes_{\OO_B}^\mathbf{L}\OO_X.\]
For a local complete intersection $i:X\to B$, the cohomology sheaves of the derived tensor product are given by exterior powers of the corresponding conormal bundle
\begin{equation}\label{eqn:cohomformality}
\mathcal{H}^{-i}(\OO_W)=\wedge^i N^\vee_{X/B}.    
\end{equation}

Consider the formal dg symmetric algebra of the conormal sheaf 
\[\Sym(N^\vee_{X/B}[1]):=\bigoplus_i \wedge^i N^\vee_{X/B}[i],\]
equipped with the standard wedge product (and 0 differentials). The map $i : X \to B$ is \emph{formal} if formula \eqref{eqn:cohomformality} can be lifted to the derived level. This can be interpreted as:

\begin{itemize}
\item An isomorphism of dg algebras
\begin{equation}\label{eqn:formalalg}
h_*\OO_W \longsimeq \Delta_*\Sym(N^\vee_{X/B}[1]) \qquad \text{in }\D(X \times X),    
\end{equation}
with the natural maps $h : W \to X \times X$, $\Delta : X \to X \times X$, or

\item An isomorphism of dg schemes
\begin{equation}\label{eqn:formalgeom}
W \longsimeq N_{X/B}[-1]    
\end{equation}
over $X \times X$. 
\end{itemize}

The formality of the derived self-intersection is governed by the first infinitesimal neighborhood $X^{(1)}$ of $X$ inside $B$. In this paper, we use the notion of \textit{quantized cycles}.

\begin{definition}
    A quantized cycle \cite{grivaux2014hochschild} is a retraction $\sigma:X^{(1)}\to X$ of the closed embedding $j:X\to X^{(1)}$ of $X$ into its first infinitesimal neighborhood.
\end{definition}

We alter Theorem 1.8 of \cite{arinkin2019formality} to the case of local complete intersections.

\begin{theorem}[{\cite[Theorem 1.8]{arinkin2019formality}}]\label{thm:formality}
    
    Let $i:X\to B$ be a regular embedding with a quantized cycle $\sigma:X^{(1)}\to X$. Then, we have the following statements.
    \begin{enumerate}
    \item The derived self-intersection is formal in the algebraic sense \eqref{eqn:formalalg}, meaning that there exists an isomorphism of dg-algebras
    \[h_*\OO_W \longsimeq \Delta_*\Sym(N^\vee_{X/B}[1]) \qquad \text{in }\D(X \times X).\]
    
    \item The derived self-intersection is formal in the geometric sense \eqref{eqn:formalgeom}: there exists an isomorphism $W \longsimeq N_{X/B}[-1]$ of dg schemes over $X\times X$.
    \item The dg endofunctors of $\D(X)$ are isomorphic
    \[i^*i_*(-)\simeq (-)\otimes \Sym(N^\vee_{X/B}[1]).\]
    \end{enumerate}
\end{theorem}

\begin{proof}
    The equivalence of (1), (2), and (3) are proven in \cite{arinkin2019formality} in the case of a closed embedding of smooth schemes. The proof is the same in the case of a local complete intersection.

    Now, we prove (3). We consider the factorization of $i:X\to B$ given by the embedding to the first infinitesimal neighborhood $j:X\to X^{(1)}$, and then to the ambient space $X^{(1)}\to B$. We have the following sequence of functors
    \[i^*i_*(-)\Rightarrow j^*j_*(-)\Rightarrow j^*j_*j^*\sigma^*(-)\simeq j^*(\sigma^*(-)\otimes j_*\OO_X)\simeq (-)\otimes j^*j_*\OO_X\]
    given by first an adjunction, then using that $\sigma \circ j=id$, then using the projection formula, and again using that $\sigma \circ j=id$. 
    
    In \cite{arinkin-caldararu}, a morphism of complexes $j^*j_*\OO_X\to \Sym(N^\vee_{X/B}[1])$ was constructed in the smooth case that can be easily adapted to the case of a local complete intersection. We obtain a natural transformation of functors
    \begin{equation}\label{eqn:pullpushfunctorisom}
    i^*i_*(-)\Rightarrow (-)\otimes \Sym(N^\vee_{X/B}[1]).    
    \end{equation}

    To show \eqref{eqn:pullpushfunctorisom} is an isomorphism, we only need to show that for every locally free sheaf $\mathcal{E}$, the natural transformation yields a quasi-isomorphism $i^*i_*\mathcal{E}\to \mathcal{E}\otimes \Sym(N^\vee_{X/B}[1])$. This is verified in Theorem A of \cite{grivaux2020derived}.
\end{proof}

We turn our attention back to log schemes and the log diagonal map $i:X\to B$.

\begin{proposition}\label{prop:formalityloghochschild}
If $X$ satisfies Assumption \ref{ass:logsmvar}, Theorem \ref{thm:formality} applies to the closed immersion $i: X\to B$. 
\end{proposition}

\begin{proof}
The map $X\to B$ is a local complete intersection morphism by Lemma \ref{lem:xtoblci}. Thus, in order to apply Theorem \ref{thm:formality}, we need to show that there exists a quantized cycle $\sigma$ for the closed embedding $X\to B$. Such a quantized cycle can be obtained by considering the sequence of maps
\[X^{(1)}\to B\to X\times X\to X\]
where the first map is the closed immersion of the first infinitesimal neighborhood in the ambient space, the second map comes from the Cartesian diagram \eqref{eq:diagram}, and the last map is one of the projections $\pi_1, \pi_2:X\times X\to X$.


\end{proof}

\begin{remark}
In the case of the log diagonal map $i:X\to B$, a global splitting $p:B\to X$ exists. The quantized cycle in the proof of Proposition \ref{prop:formalityloghochschild} is the restriction of such a splitting to the first infinitesimal neighborhood.
\end{remark}

\begin{remark}\label{rmk:noneedlogseparated}

If $X$ is not weakly log separated, $X \to B$ is not a closed immersion. But Lemma \ref{lem:xtoblci} still shows the cotangent complex of $X \to B$ is concentrated in degree $-1$. It is an unramified local complete intersection, which is étale locally a regular immersion by \cite[04HH]{sta}. 

We suspect Theorem \ref{thm:formality} still holds for $X \to B$, with $N^\vee_{X/B}$ replaced by the single nonzero cohomology group of its cotangent complex $h^{-1}(\ccx{X/B})$. The difficulty is that there is no first infinitesimal neighborhood $X^{(1)}$, necessary for constructing the natural transformation of functors. Since the statements in Theorem \ref{thm:formality} and Proposition \ref{prop:formalityloghochschild} are local, this should not matter. 

We leave these technical points to the reader. 
    
\end{remark}

As a corollary of Proposition \ref{prop:formalityloghochschild}, we get an explicit expression for the log Hochschild co/homology of $X$.

\begin{corollary}\label{cor:formalityloghochschild}
For $X$ satisfying Assumption \ref{ass:logsmvar}, we have natural isomorphisms of functors
\[i^*i_*(-) \longsimeq - \otimes \Sym(\lkah{X}[1])\]
and
\[i^!i_*(-)\longsimeq - \otimes \Sym(T^{\log}_X[-1]).\]

\end{corollary}

\begin{proof}

Using Proposition \ref{prop:formalityloghochschild} and Theorem \ref{thm:formality}, we have an isomorphism of functors
\[ i^* i_* (-) \simeq - \otimes \Sym (N^\vee_{X/B}[1]).\]

This statement is given by rewriting the previous one through the natural identifications
\[\lkah{X} \simeq N^{log \vee}_{X/X \times X} \simeq N^\vee_{X/B}\]
\[T^{\log}_X \simeq N^{\log}_{X/X \times X} \simeq N_{X/B}.\]
For a regular embedding $X\to B$ of codimension $d$, we have that the functor $i^!$ is isomorphic to the functor $i^*(-)\otimes \wedge^{d} N_{X/B}[-d]$ implying that 
\[i^!i_*(-)\longsimeq - \otimes \Sym(T^{\log}_X[-1]).\]

\end{proof}

Applying Corollary \ref{cor:formalityloghochschild} to the structure sheaf we obtain the log version of the classical HKR isomorphism \cite{kostant2009differential}.

\begin{theorem}[Log HKR]\label{thm:loghkr}
If $X$ satisfies Assumption \ref{ass:logsmvar}, the log Hochschild co/homology of $\OO_X$ is computed in terms of the log co/tangent bundles, namely, there exist isomorphisms of $k$-vector spaces 
\[R^n\Gamma(X,\HHl{X}(\OO_X))=\bigoplus_{q-p=n}H^p(X,\Omega^{q,\log}_X)\]
and
\[R^n\Gamma(X,\cHHl{X}(\OO_X))=\bigoplus_{p+q=n}H^p(X,\wedge^q T^{\log}_X).\]
\end{theorem}

\begin{example}
    Consider $X=\PP^1$ with log structure given by an effective Cartier divisor of $n$ distinct points. In this case, we have that $\Omega^{1,\log}_X=\OO_{\PP^1}(n-2)$, and thus if $n\geq 2$, we have that the 1st Hochschild homology space is of positive dimension,
    \[\dim R^1\Gamma(X, \HHl{X}(\OO_X))>0.\]
    This result is in contrast to the case of ordinary Hochschild homology of $\PP^1$, which is concentrated in degree 0: 
    \[HH_i(\PP^1)=0 \qquad i\ne 0.\]
\end{example}

Formality also shows log Hochschild homology is invariant under log alterations.

\begin{theorem}\label{thm:logHHlogaltn}

Let $\pi : X \to Y$ be a log étale map between log smooth schemes satisfying Assumption \ref{ass:logsmvar}. The log Hochschild co/homology of an object $F \in \D(Y)$ pulls back to that on $X$:
\[
\pi^* \HHl Y (F) = \HHl X (\pi^* F), \qquad \pi^* \cHHl Y (F) = \cHHl X (\pi^* F).
\]
    
\end{theorem}

\begin{proof}

The log tangent and cotangent bundles pull back
\[\pi^* \lkah{Y} = \lkah{X}, \qquad \pi^* \Tl{Y} = \Tl{X}.\]
Apply Corollary \ref{cor:formalityloghochschild} and use compatibility of $\pi^*$ with the tensor product and symmetric algebra functors. 
    
\end{proof}

\subsubsection{The nodal cubic}

Let $X \subseteq \PP^2$ be the nodal cubic cut out by the equation
\[y^2z = x^3 + x^2z.\]
We compute relative log Hochschild homology for this example and contrast with ordinary Hochschild.


The nodal cubic $X$ admits a unique log structure making it an integral, saturated, vertical, log smooth curve over a point $v$ with rank one log structure $\bar M_v = \NN$. 
If $p \in X$ is the node $[0 : 0 : 1]$, the stalk $\bar M_{X, p} = \NN^2$ has rank two, but all the other points $x \in X$ have rank one $\bar M_{X, x} = \NN$\footnote{
This is locally the same as equipping $\PP^2$ with divisorial log structure along $X$ and then restricting the log structure to $X$. But that log structure would encode the nontrivial normal bundle $N_{X/\PP^2}$ of the nodal cubic, as opposed to the trivial bundle pulled back from $t \in \Gamma(\gp{\bar M}_v)$. 
}. 

The map on characteristic monoids from $X \to v$ is
\[\bar M_v \longsimeq \bar M_x, \qquad \bar M_v \to \bar M_p; \quad t \mapsto (t, t)\]
at the node $p$ and general point $x \in X$. Write $q \in \Gamma(v, \gp M_v)$ for an element mapping to $1 \in \NN = \Gamma(v, \gp{\bar M}_v)$.

Étale locally around $p$, $X$ is isomorphic to the union of the axes $V(xy) \subseteq \Aff^2$. Thus its tropicalization has a cover by $\af{\Aff^2} = \af{}^2$. As you traverse the loop in $X$, the $x-$ and $y-$axes turn out to be a single divisor, corresponding to a single ray of $\af{X}$. So the tropicalization $\af{X}$ of $X$ is the quotient of $\af{}^2$ by identifying the $x-$ and $y-$axes; it looks like a waffle cone. See Figure \ref{fig:tropicalnodalcubic}.

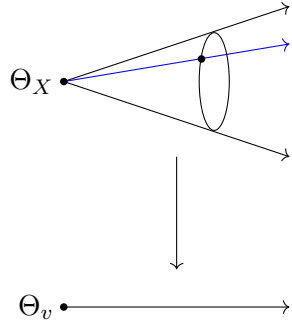
\begin{figure}
    \centering
    \begin{tikzpicture}
    \draw[->] (1.5, -1) to (1.5, -2.5);
    %
    %
    \draw[->] (0, 0) to (3, 1);
    \draw[->] (0, 0) to (3, -1);
    \draw[->, blue] (0, 0) to (3, .5);
    \fill (0, 0) circle (.05);
    \node[left] (0, 0) {$\af{X}$};
    \fill (1.83, .3) circle (.05);
    \draw(2,0) ellipse (.2cm and .65cm);
    \begin{scope}[shift = {(0, -3)}]
    \draw[->] (0, 0) to (3, 0);
    \fill (0, 0) circle (.05);
    \node[left] (0, 0) {$\af{v}$};
    \end{scope}
    %
    \end{tikzpicture}
    \caption{
    The tropicalization of the nodal cubic $\af{X}$ is the cone over a circle with a single vertex. It is obtained from $\af{}^2$ by gluing together the two axes. It has one ray, depicted in blue. 
    }
    \label{fig:tropicalnodalcubic}
\end{figure}

The log Hochschild homology of $X \to v$ is computed by Theorem \ref{thm:loghkr}, once we compute the log Kähler differentials $\lkah{X}=\OO_X$.

\begin{lemma}\label{lem:nodalcubiclogcotangent}

The log cotangent bundle of $X$ is trivial
\[\lkah{X} \simeq \OO_X.\]
    
\end{lemma}


\begin{proof}

The log cotangent bundle for nodal curves is the dualizing sheaf \cite[Proposition 1.13]{fkatomodulilogcurves}
\[\omega_X = \lkah{X}.\]
Under the normalization map $\nu : \PP^1 \to X$, the dualizing sheaf pulls back to
\[\nu^* \omega_X = \lkah{\PP^1}(2) = \OO.\]
It is of degree zero, giving a section of $\Pic^0_X$. 

Consider a flat family $\tilde X \to \Aff^1$ of smooth genus-one curves degenerating to $\tilde X \times_{\Aff^1} \vec 0 = X$. For example, take the Legendre family
\[V\left(y^2z = x(x+z)(x+\lambda z)\right) \qquad \subseteq \PP^2 \times \Aff^1_t\]
and restrict to $\Aff^1_t \setminus \{1\}$. The two line bundles $\lkah{\tilde X/\Aff^1}$ and $\OO_{\tilde X}$ on $\tilde X$ give sections of $\Pic^0_{\tilde X/\Aff^1}$ over $\Aff^1$ that agree away from the origin $\vec 0 \in \Aff^1$. The sheaf $\Pic^0_{\tilde X/\Aff^1}$ is representable by a separated scheme \cite[Theorem 3, \S 8.4]{neronmodels}, so the two sections coincide.

\end{proof}

Explicitly, we have 
\[\dim R^i\Gamma(X,\HHl{X}(\OO_X))=
\begin{cases}
2 & i=0\\
1 & i=\pm 1\\
0 & \text{otherwise},
\end{cases}\]
agreeing with the ordinary Hochschild homology of a smooth cubic curve.

The ordinary Hochschild homology of the map $X^\circ \to v^\circ = \pt$ is quite complicated with infinitely many non-zero terms. 

\subsection{Duflo isomorphism}\label{ss:duflo}
Both sides of both statements of Theorem \ref{thm:loghkr} have natural algebra structures. 

The algebra structure on $R\Gamma(X,\HHl{X}(\OO_X))$ is given by the natural algebra structure on $i^*i_*\OO_X$ and the algebra structure on $\oplus_{q-p=\star}H^p(X,\Omega^{q,\log}_X)$ is given by the usual wedge product on $\Sym(\Omega^{1,\log}_X[1])$. It was proven in \cite{arinkin-caldararu} that these two algebra structures coincide in the smooth case, even on the level of the dg algebras $i^*i_*\OO_X$ and $\Sym(\Omega^{1,\log}_X[1])$. 
The proof can be adapted to the case of local complete intersections, showing that the log Hochschild homology as a dg algebra is isomorphic to the dg algebra of logarithmic 1-forms, $\oplus_{q-p=\star}H^p(X,\Omega^{q,\log}_X)$.

The algebra structure on log Hochschild cohomology $R^\star\Gamma(X,\cHHl{X}(\OO_X))$ comes from the Yoneda product on the derived Hom space $\RHom_B(i_*\OO_X, i_*\OO_X)$ while the algebra structure on the dg algebra of logarithmic polyvector fields $\oplus_{p+q=\star}H^p(X,T^{q,\log}_X)$ is given by the usual wedge product. Contrary to the case of log Hochschild homology, these algebra structures may differ. For ordinary Hochschild cohomology \cite{kontsevich2003deformation, calaque2010hochschild}, the algebra structures become isomorphic after contracting with the square root of the Todd class.

For log Hochschild cohomology, we likewise conjecture the algebra structures are isomorphic after contraction with the square root of the ``log Todd class.'' To formulate our conjecture, we first define what we mean by the log Todd class. We follow \cite{grivaux2014hochschild, yu2019todd}.

\subsubsection{Logarithmic Todd class}

Since the map $i:X \to B$ is a local complete intersection morphism of codimension $d$, Verdier duality tells us that
\[\wedge^d T_X^{\log}[-d]\simeq \RSHom_X(\OO_X, i^!\OO_B).\]
Using the natural map $\OO_B\to i_*\OO_X$, we obtain a map
\[\RSHom_X(\OO_X, i^!\OO_B)\to \RSHom_X(\OO_X, i^!i_*\OO_X)\]
where the latter (using Verdier duality again) can be identified with
\[\RSHom_X(\OO_X, i^!i_*\OO_X)\simeq \RSHom_X(\OO_X, i^*i_*\OO_X\otimes \wedge^d T_X^{\log}[-d])\simeq \Sym(T_X^{\log}[-1])\]
where the last isomorphism comes from the log HKR map of Corollary \ref{cor:formalityloghochschild}. The composite of these maps is a map
\[\wedge^d T_X^{\log}[-d]\to \Sym(T_X^{\log}[-1])\]
in the (bounded) derived category $\D^b(X)$. We call the corresponding element in
\[\Hom_{\D^b(X)}(\wedge^d T_X^{\log}[-d], \Sym(T_X^{\log}[-1]))=R^0\Gamma(X, \Sym(\Omega^{1,\log}_X[1]))=\]
\[=\bigoplus_n H^n(X, \Omega^{n,\log}_X)\]
the \emph{log Todd class} of $X$. We denote this class by $Td_X^{\log}$.

\begin{remark}
Via the logarithmic Hodge-to-de Rham degeneration \cite{kato1989logarithmic, esnault1985logarithmic, hablicsek2020hodge}, we see that the Todd class can be interpreted as a class in the even part of the logarithmic Hodge cohomology, $\bigoplus_n H^{2n, \log}_{dR}(X)$ as one would expect \cite{scherotzke2020parabolic}.
\end{remark}

\begin{remark}
The degree 0 part of the Todd class, $Td_X^{\log}$ is always 1, it comes from the identification of the degree $d$ part of $\Sym(T_X^{\log}[-1])$ with $\wedge^d T_X^{\log}$, see \cite{grivaux2014hochschild}. As a consequence, it makes sense to consider the square root of the Todd class.
\end{remark}

Now, we are ready to state our conjecture.

\begin{conjecture}
    The composite morphism $I:HKR^{-1}\circ \iota_{\sqrt{Td_X^{\log}}}$ provides an isomorphism of dg-algebras $\oplus_{p+q=\star}H^p(X,\wedge^q T^{\log}_X)\to R^\star\Gamma(X,\cHHl{X}(\OO_X))$ where $HKR^{-1}$ is the inverse of the composite 
    \[HKR:\Ext^\star_B(i_*\OO_X, i_*\OO_X)\longsimeq \Ext^\star_X(i^*i_*\OO_X,\OO_X)\longsimeq \]
    \[\longsimeq\oplus_k\Ext^{\star-k}(\Omega^{k,\log}_X, \OO_X)\longsimeq \oplus_{p+q=\star}H^p(X,\wedge^q T^{\log}_X)\]
    induced by adjunction and the formality isomorphism of Corollary \ref{cor:formalityloghochschild}, and $\iota_{\sqrt{Td_X^{\log}}}$ is contraction with the square root of the log Todd class.
\end{conjecture}


\subsection{Smooth pairs $(X, D)$}

Our construction can be described explicitly when the log structure on $X$ is given by a smooth effective Cartier divisor $D$. Then $B$ is an open subset of the blow-up $Bl_{D\times D}(X\times X)$ of $D\times D$ inside $X\times X$. 

The diagonal map $\Delta:X\to X\times X$ factors through $Bl_{D\times D}(X\times X)$, providing Cartesian squares
\[\begin{tikzcd}
D\ar[d] \pb \ar[r] & X\ar[d, "j"]\\
E\ar[r]\ar[d] \pb & Bl_{D\times D}(X\times X)\ar[d]\\
D\times D\ar[r]& X\times X.
\end{tikzcd}
\]
We denote the map $X\to Bl_{D\times D}(X\times X)$ by $j$. Since $B$ is an open subset of $Bl_{D\times D}(X\times X)$, we have a natural isomorphism of functors $i^*i_*$ and $j^*j_*$. For the same reason, $N^\vee_{X/Bl_{D\times D}(X\times X)} \simeq N^\vee_{X/B} \simeq \Omega^{1,log}_X$.

Recall $(X, D)$ is weakly log separated if and only if $X^\circ$ is separated by Proposition \ref{prop:logsepsnc}. For pairs, one could define log Hochschild co/homology and repeat our proofs with the classical blow-up $Bl_{D \times D} X \times X$ instead of the log diagonal $B$. To summarize the above discussion, we have the following statement which is an explicit version of Theorem \ref{thm:loghkr}.

\begin{proposition}

For a pair $(X,D)$ of a smooth, separated scheme $X$ and $D$ a simple normal crossing divisor on $X$, we obtain
\[R^n\Gamma(X,\HHl{X}(\OO_X))=\oplus_{q-p=n}H^p(X,\Omega^{q}_X(\log D))\]
and
\[R^n\Gamma(X,\cHHl{X}(\OO_X))=\oplus_{p+q=n}H^p(X,\wedge^q T_X(-\log D)).\]

\end{proposition}

\begin{proof}

Omitted. 
    
\end{proof}

\begin{example}\label{ex:logdiagonalaffineline}

Consider the pair $(X, D) = (\Aff^1, \vec{0})$. The blow-up $Bl_{D \times D} X \times X = Bl_{\vec 0} \Aff^2$ is the blow-up of the affine plane at the origin. The Artin fans of $\Aff^1, \Aff^2$ are their quotients by dense tori $\bra{\Aff^1/\GG_m}, \bra{\Aff^1/\GG_m}^2$ as in Figure \ref{fig:tropicalblowupA1}. We have a Cartesian diagram
\[
\begin{tikzcd}
X \ar[r] \ar[dr]      &B \ar[r] \ar[d] \lpbstrict      &Bl_{\vec 0} \Aff^2 \ar[r] \ar[d] \lpbstrict         &\Aff^2 \ar[d]         \\
        &\bra{\Aff^1/\GG_m} \ar[r]         &\Sigma \ar[r]         &\bra{\Aff^2/\GG_m^2}.
\end{tikzcd}
\]
Here $\Sigma$ is the stack quotient of $Bl_{\vec 0} \Aff^2$ by its dense torus, the Artin fan corresponding to the first quadrant subdivided along the diagonal. 

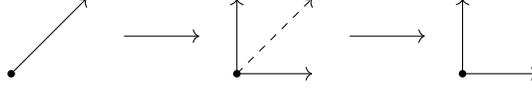
\begin{figure}
    \centering
    \begin{tikzpicture}
    \draw[->] (0, 0) to (1, 1);
    \fill (0, 0) circle (.05);
    \draw[->] (1.5, .5) to (2.5, .5);
    \draw[->] (4.5, .5) to (5.5, .5);
    \begin{scope}[shift = {(3, 0)}]
    \draw[->] (0, 0) to (1, 0);
    \draw[->] (0, 0) to (0, 1);
    \draw[->, dashed] (0, 0) to (1, 1);
    \fill (0, 0) circle (.05);
    \end{scope}
    \begin{scope}[shift = {(6, 0)}]
    \draw[->] (0, 0) to (1, 0);
    \draw[->] (0, 0) to (0, 1);
    \fill (0, 0) circle (.05);
    \end{scope}
    \end{tikzpicture}
    \caption{The cone complexes corresponding to the Artin fans $\bra{\Aff^1/\GG_m} \to \Sigma \to \bra{\Aff^1/\GG_m}^2$. These are the tropicalizations of $X \to Bl_{\vec 0} \Aff^2 \to \Aff^2$. The blowup is produced by subdividing $\af{\Aff^2}$ along the image of the diagonal $\af{\Aff^1}$. We suspect this blowup description exists in general. }
    \label{fig:tropicalblowupA1}
\end{figure}

Write $E' \simeq \PP^1$ for the exceptional divisor in $Bl_{\vec 0} \Aff^2$ and $E \simeq \GG_m \subseteq E'$ for the complement of its two points of intersection with the strict transforms of the $x-$ and $y-$ axes. The space $B$ in this case is the complement in $Bl_{\vec 0} \Aff^2$ of the strict transforms of the axes 
\[B = \Aff^2 \setminus V(xy) \cup E.\] 
This is obtained by decomposing the blow-up into its $\GG_m^2$-invariant components
\[Bl_{\vec 0} \Aff^2 = (\Aff^2 \setminus Z(xy)) \cup (Z(x) \setminus \{0\}) \cup (Z(y) \setminus \{0\}) \cup E\]
and taking the components to which the diagonal maps.

\end{example}

\begin{example}

Let us apply theorem \ref{thm:loghkr} to compute the log Hochschild homology and cohomology of $\Aff^1$ with the log structure induced by $0$. Recall that $\Omega^{1,log}_{\Aff^1} \simeq k[t] \dfrac{dt}{t}$ and dually $T^{\log}_{\Aff^1} \simeq t \cdot k[t] \dfrac{\partial}{\partial t}$ are locally free of rank 1. 
The complexes obtained by taking wedge products have no higher terms. Thus the log Hochschild homology (respectively cohomology) is concentrated in degrees 0 and 1 (respectively 0 and $-1$).
    
\end{example}

One may ask if $B$ always appears as an open of a log blowup of $X \times X$. 

\begin{example}\label{ex:A2logdiagonalblowup}

Let $X = \Aff^2$. The diagonal map on Artin fans for $\Aff^2 \to \Aff^2 \times \Aff^2$ embeds the Artin fan $\af{\Aff^2} = \bra{\Aff^2/\GG_m^2}$ inside its diagonal. This is the cone over a line segment inside the tetrahedron, depicted in blue in Figure \ref{fig:A2diagonal}. 

Adding the blue line segment yields a subdivision $\Sigma \to \af{\Aff^2 \times \Aff^2} = \af{}^4$ into non-convex cones. The pullback $\tilde X \coloneqq \Aff^4 \times_{\af{\Aff^4}} \Sigma$ is not representable by a scheme, though it has a functor of points on log schemes. It is the analogue of $Bl_{\vec 0} \Aff^2$ in the example of $\Aff^1$ because $B_{\Aff^2} \subseteq \tilde X$ is an open. 

The map $\tilde X \to \Aff^4$ is like a ``log blowup'' of $\Aff^4$. There is a further refinement of $\Sigma$ by convex cones that yields a representable log modification $Y$ both of $\Aff^4$ and of $\tilde X$. The space $\tilde X$ is a \emph{logarithmic space} (of the second kind) \cite{logabelianvarieties2}, \cite{logpic}.

There are many such choices of $Y$. For example, one can take the minimal common refinement of the star subdivisions at the two endpoints of the line segment. 

One could also take the barycentric subdivision. But this involves subdividing also along the blue line segment, modifying the diagonal $\af{\Aff^2}$.

\begin{figure}
    \centering
    \begin{tikzpicture}
    \draw[-] (0, 0) to (1, 1.7) to (2, 0) to (2.5, 1) to (0, 0);
    \draw[-] (1, 1.7) to (2.5, 1);
    \draw[-] (0, 0) to (2, 0);
    \draw[-, blue, dashed] (.5, .85) to (2.25, .5);
    \filldraw (.5, .85) circle (.1em);
    \filldraw (2.25, .5) circle (.1em);
    \end{tikzpicture}
    \caption{The diagonal $\af{\Aff^2} \to \af{\Aff^2} \times \af{\Aff^2}$ is the cone over the inclusion of the blue dashed line segment into the tetrahedron. Subdividing at the image of the diagonal adds the blue line to obtain a non-convex subdivision $\Sigma$ of $\af{\Aff^2 \times \Aff^2} = \af{}^4$. 
    This non-convex subdivision has a functor of points on log schemes and admits a refinement by a subdivision with convex cones.  }
    \label{fig:A2diagonal}
\end{figure}

\end{example}

\subsection{Cyclic homology of a logarithmic scheme}\label{sec:cyclic}

We develop the log version of cyclic homology parallel to the non-logarithmic approach \cite{toen2007note, ben2012loop, bokstedt1993topological}. We begin with the definition of the logarithmic loop space.

\begin{definition}
    Let $X$ be a log scheme with Artin fan $X\to \af{X}$. We define the logarithmic derived loop space of $X$ as the derived stack
    \[LX^{\log} :=\Map_{\dSt/\af{X}}(S^1, X).\]
\end{definition}
Here, $S^1$ denotes the stack associated to the constant simplicial presheaf
\[\alg/\af{X}\to \SSets.\]
Since the homotopy pushout of
\[
\begin{tikzcd}
\star \sqcup \star \ar[r]\ar[d]& \star\\
\star &
\end{tikzcd},
\]
is $S^1$, we obtain that the logarithmic derived loop space can be realized as a derived self-intersection
\[LX^{\log}=X\times^h_{X\times^h_{\af{X}}X} X.\]

If $X$ is log flat, the derived and underived log fibre products coincide 
\[X\times^h_{\af{X}}X \simeq B.\] 
The logarithmic derived loop space of a log smooth scheme $X$ is the derived self-intersection
\[LX^{\log}\simeq X\times^h_BX.\]
Using either of the projection maps $p:LX^{\log}\to X$ and Proposition 1.4 of \cite{toen2012proper}, we have
\[p_*\OO_{LX^{\log}}=i^*i_*\OO_X=\HHl{X} (\OO_X).\]
For log smooth $X$, we have identified the functions on the logarithmic derived loop space with the logarithmic Hochschild homology of $X$ evaluated at the structure sheaf.

The logarithmic derived loop space is equipped with an $S^1$-action induced by the action on the first argument $\Map_{\dSt/\af{X}}(S^1, X)$ called \emph{rotating the loops}. Therefore $R\Gamma(X, \HHl{X}(\OO_X))$ can be regarded as an $S^1$-equivariant complex, in other words, as a mixed complex \cite{kassel1987cyclic, loday1998cyclic}.

We are ready to define the versions of the cyclic homology of a log scheme.
\begin{definition} For a log smooth scheme $X$, we define
\begin{itemize}
    \item the $i$-th cyclic homology as the complex of homotopy orbits of the $S^1$-action
    \[HC^{\ell}_i(X):=\pi_i(R\Gamma(X, \HHl{X}(\OO_X))_{hS^1}),\]
    \item the $i$-th negative cyclic homology as the complex of homotopy fixed points of the $S^1$-action
    \[HC^{-,\ell}_i(X):=\pi_i(R\Gamma(X, \HHl{X}(\OO_X))^{hS^1}),\]
    \item and finally, the $i$-th periodic cyclic homology as the complex of Tate fixed points of the $S^1$-action
    \[HC^{per,\ell}_i(X):=\pi_i(R\Gamma(X, \HHl{X}(\OO_X))^{tS^1}),\]
\end{itemize}
    
\end{definition}

\begin{remark}
    M. Olsson defines an isomorphic ``log cyclic co/homology'' using explicit chain complexes \cite{olssondraftloghochschild}.
\end{remark}
 
From the universal property of the pullback, we obtain a natural $S^1$-equivariant map $X\to LX^{\log}$ called the constant loop map. The structure morphism $\af{X}\to k$ induces a forgetful map $\mathcal{F}_{LX}:LX^{\log}\to LX$ from the logarithmic derived loop space to the loop space of $X$ forgetting the log loops.
\[
\begin{tikzcd}
    LX \ar[ddr, red]\ar[rrd, red] & & &\\
    & LX^{\log}\ar[lu, dotted] \ar[r]\ar[d] & X\ar[d]\ar[rdd, red] &\\
    & X \ar[rrd, red]\ar[r] & B\ar[rd] &\\
    & & & X\times X
\end{tikzcd}
\]

This forgetful map is $S^1$-equivariant, providing an $S^1$-equivariant map of complexes (thus, a map of mixed complexes)
\begin{equation}
    R\Gamma(X, \HH{X}(\OO_X))\to R\Gamma(X, \HHl{X}(\OO_X)).
\end{equation}

\begin{example}
In the case of a log scheme given by a pair $(X,D)$ with $X$ a smooth scheme and $D$ a simple normal crossing divisor on $X$, the map of mixed complexes
\[R\Gamma(X, \HH{X}(\OO_X))\to R\Gamma(X, \HHl{X}(\OO_X))\]
induced by the forgetful map $\mathcal{F}_{LX}$ is given by the inclusion $\Omega^1_X\to \Omega^1_{X}(\log D)$. 

The degree $-1$ differential on the mixed complex $R\Gamma(X, \HHl{X}(\OO_X))$ is given by the 
log de Rham differential as in the case of $R\Gamma(X, \HH{X}(\OO_X))$. As a consequence, the variants of the cyclic homology can be computed in terms of the logarithmic de Rham complex. 

For instance, the periodic cyclic homology computes the even (or odd) part of the logarithmic de Rham cohomology, 
\[HC^{per,\ell}_0(X)=\prod_{i\geq 0}H^{2i}_{DR,log}(X),\]
and
\[HC^{per,\ell}_1(X)=\prod_{i\geq 0}H^{2i+1}_{DR,log}(X).\]

Using the logarithmic Hodge-to-de Rham degeneration \cite{kato1989logarithmic, esnault1985logarithmic, hablicsek2020hodge}, 
we can compute the periodic cyclic homology explicitly:
\[HC^{per,\ell}_0(X)=\bigoplus_{p+q \text{ is even}} H^{q}(X, \Omega^p_X(\log D)),\]
and
\[HC^{per,\ell}_1(X)=\bigoplus_{p+q \text{ is odd}} H^{q}(X, \Omega^p_X(\log D)).\]

\end{example}

\section{Application to orbifolds}\label{s:logorbifolds}

In this section, we describe the Hochschild co/homology of log (quotient) orbifolds $[X/G]$ in the case of a \textit{firm} action (see Assumption \ref{ass:actionstrata}). Using formality results, we understand the logarithmic derived loop space associated to a log orbifold and we derive formulae for the Hochschild co/homology. We closely follow \cite{arinkin2019formality}, so the proofs are only sketched. 

Let $G$ be a finite group acting on $(X,M_X)$ \textit{as a log scheme}: the action map $G \times X \to X$ is a map of log schemes. We view $G$ as a finite étale constant group scheme with trivial log structure. We work over a field $k$ in which the order of $G$ is invertible. 


\begin{remark}
    Let $X$ be a smooth scheme equipped with divisorial log structure from a normal crossings divisor $D$. 
    The action respects the log structure if and only if it restricts to $D$. Assume that there exists an element $g \in G$ and a point $x \in D$ such that $g.x \notin D$. Then the induced map $(g.^{-1} j_* \OO_U)_x \to (j_* \OO_U)_x$ (here $U = X \setminus D)$) has a non-zero source and a zero target, thus it is not bijective.
\end{remark}



\begin{remark}

The main technical difficulty here is that the Artin fan of an orbifold is not isomorphic to the corresponding orbifold of Artin fans. More precisely, consider the stack quotient $Y = \bra{X/G}$ with its natural log structure. The maps
\[G \times X \rightrightarrows X \to Y\]
are strict, as is each map
\[g : X \to X\]
by which $g$ acts on $X$. These maps induce maps on Artin fans
\[\af{G \times X} = G \times \af{X} \rightrightarrows \af{X} \to \af{Y}.\]

This diagram yields a map
\[\bra{\af{X}/G} \to \af{Y}\]
which one might expect to be an isomorphism. Unfortunately, it never is unless $G$ is trivial. The vertex of $\af{X}$ is always a fixed point for the $G$-action. The Artin fan $\af{Y}$ is always representable over $\Log$ by construction, whereas $\bra{\af{X}/G}$ need not be. But the map is a partial coarse moduli space, in particular finite and flat. 

\end{remark} 

To circumvent this difficulty, we consider log orbifolds where the action is \textit{firm}:

\begin{assumption}\label{ass:actionstrata}
For the rest of this article, we assume the projection and action maps to the Artin fan
\[G \times X \rightrightarrows X \to \af{X}\]
coincide. Equivalently, we assume that the action of $G$ on the Artin fan is trivial (the map $X \to \af{X}$ is $G$-invariant). We call such action a \textit{firm} action, and the corresponding orbifold $[X/G]$ a \textit{firm} orbifold.
\end{assumption}

For firm actions, we have $\bra{\af{X}/G} = \af{X} \times BG$. 

\begin{remark}
    If the log structure is induced by a normal crossing divisor $D$, a firm action means that $G$ acts on each stratum of the log structure separately. 
\end{remark}

\begin{example}\label{ex:P1notbystrata}
Consider $\PP^1$ equipped with the log structure induced by the divisor that is the union of $0 = [0 : 1]$ and $\infty = [1 : 0]$, and let the nontrivial element $-1$ in $G \coloneqq \ZZ/2$ act on $\PP^1$ by $[x : y] \mapsto [y: x]$. This is the inversion action $z \mapsto \dfrac{1}{z}$. It swaps $0 = [0 : 1]$ and $\infty = [1 : 0]$. The action is not firm.

The lack of firmness can be seen tropically. The action on $\PP^1$ induces an action of $G$ on the tropicalization $\af{\PP^1}$. The tropicalization $\af{\PP^1} \subseteq \af{}^2$ is not invariant under the group action as the action swaps the coordinates on $\af{}^2$. We illustrate this with Figure \ref{fig:P1notbystrata}.

\begin{figure}
    \centering
    \begin{tikzpicture}
    \draw[-] (-1, 0) to (1, 0);
    \filldraw (0, 0) circle (.08);
    \node at (-1, 1){$\af{\PP^1}$};
    \draw[<->, bend left] (-.2, .2) to (.2, .2);
    \begin{scope}[shift = {(4, 0)}]
    \draw[-] (-1, 0) to (1, 0);
    \draw[-] (0, -1) to (0, 1);
    \filldraw (0, 0) circle (.08);
    \draw[-, dashed] (-1, -1) to (1, 1);
    \draw[-, dashed] (-1, 1) to (1, -1);
    \node at (-1, 1){$\Delta_{-1}$};
    \node at (1, 1){$\Delta$};
    \end{scope}
    \end{tikzpicture}
    \caption{The action of $G = \ZZ/2$ on the tropicalization $\af{\PP^1}$ of $\PP^1$ swaps the two rays. 
    The tropicalization of the diagonal map $\af{\PP^1} \to \af{\PP^1 \times \PP^1}$ is different from that of the twisted diagonal $\Delta_{-1}(x) = (x, -1.x)$.
    Here it is not enough to subdivide at the usual log diagonal $\Delta$; one must also divide at the variants $\Delta_g(x) = (x, g.x)$ for all $g \in G$. 
    This action does \textit{not} satisfy Assumption \ref{ass:actionstrata}.
    }
    \label{fig:P1notbystrata}
\end{figure}    
\end{example}

\begin{lemma}\label{lem:firmAF}

If an action $G \action X$ is firm, the Artin fan of the quotient stack $\bra{X/G}$ is the same as that of $X$:
\[\af{X} \longsimeq \af{\bra{X/G}}.\]

\end{lemma}

\begin{proof}

Because $X \to \af{X}$ is $G$-invariant, it factors through $\bra{X/G}$. We claim the factorization $\bra{X/G} \to \af{X}$ is the Artin fan of the source, checking its universal property. For any factorization
\[
\begin{tikzcd}
X \ar[r] \ar[dr]      &\bra{X/G} \ar[r] \ar[d]       &\scr B \ar[d]      \\
        &\af{X} \ar[r] \ar[ur, dashed, "\exists !"]      &\Log
\end{tikzcd}
\]
through an étale representable map $\scr B \to \Log$, there is a unique dashed arrow. This results from the analogous universal property for $X \to \af{X}$ and $G$-invariance. 


\end{proof}

In the case of a firm action, the log diagonal $X\to B$ records not only the image of the diagonal $X \to X \times X$, but also the images of the ``twisted diagonals'' $\Delta_g : X \to X \times X, x \mapsto (x, g.x)$.

\begin{lemma}\label{lem:actiononB}

Let $[X/G]$ be a firm log orbifold. There is an induced action of $G\times G$ on the scheme $B = B_X$ and the natural map $B\to X\times X$ is $G\times G$-equivariant.

\end{lemma}

\begin{proof}

We sketch the proof. The action of $G\times G$ on $B$ is induced by the pullbacks 
\[
\begin{tikzcd}
G \times G \times B \ar[r] \ar[d, shift right = 1]\ar[d, shift left = 1] \lpbstrict        &G \times G \times X \times X \ar[d, shift right = 1]\ar[d, shift left = 1]       \\
B \ar[r] \ar[d] \lpbstrict        &X \times X \ar[d]         \\
\af{X} \ar[r]        &\af{X} \times \af{X}
\end{tikzcd}.
\]
where the two right vertical arrows $G\times G\times X\times X\to X\times X$ are given by the action and the projection maps. Since $[X/G]$ is firm, the composition $G\times G\times X\times X\to X\times X\to \af{X}\times \af{X}$ does not depend on whether we use the action map or the projection map. This means that the Cartesian product of $G\times G\times X\times X\times_{\af{X}\times \af{X}}\af{X}$ is $G\times G\times B$ fitting in the diagram drawn above. Thus we get a map $G\times G\times B\to B$ compatible with the action map $G\times G\times X\times X\to X\times X$, and it is straightforward to show that the map $G\times G\times B\to B$ induces a group action of $G\times G$ on $B$.
\end{proof}

As a consequence of the lemma above, we obtain a map of log orbifolds $[B/G\times G]\to [X\times X/G\times G]$. 

\begin{lemma}\label{lem:orbifoldetale}
    The map $[B/G\times G]\to [X\times X/G\times G]$ is log \'etale.
\end{lemma}

\begin{proof}

In the Cartesian diagram of Lemma \ref{lem:actiononB}, the maps $B\to X\times X$ and $G\times G\times B\to G\times G\times X\times X$ are log \'etale. The statement follows from this.  
    
\end{proof}

The commutative diagram of the comparison of the group actions of $G$ on $X$ and of $G\times G$ on $X\times X$
\[
\begin{tikzcd}
G \times X \ar[r, "\Delta_G\times \Delta_X"] \ar[d, shift right = 1]\ar[d, shift left = 1]       &G \times G \times X \times X \ar[d, shift right = 1]\ar[d, shift left = 1]       \\
X \ar[r, "\Delta_X"] \ar[d]       &X \times X \ar[d]         \\
\af{X} \ar[r]        &\af{X} \times \af{X}
\end{tikzcd}.
\]
provides a map $G\times X\to G\times G\times B$, and subsequently a natural map of orbifolds
\begin{equation}\label{eqn:logdiagorbifold}
h:[X/G]\to [B/G\times G].
\end{equation}
Lemma \ref{lem:logdiagorbifold} will show that there is no confusion in calling it (again) log diagonal.


\begin{lemma}\label{lem:logdiagorbifold}
    There is an isomorphism of logarithmic algebraic stacks
    \[ [B / G \times G] \longsimeq [X/G] \times_{\af{X}} [X/G].\]
\end{lemma}

\begin{proof}

By Assumption \ref{ass:actionstrata}, there is an isomorphism
\[ [B / G \times G] = [X \times_{\af{X}} X / G \times G] \simeq [X / G] \times_{\af{X}} [X/G].\]
By Lemma \ref{lem:firmAF} we have
\[[X / G] \times_{\af{X}} [X/G] \simeq [X / G] \times_{\af{[X/G]}} [X/G].\]


\end{proof}

This ensures that the log Hochschild homology of a firm orbifold is equivalent to the endofunctor $h^* h_*$ where
\[ h : [X/G] \to [B/ G \times G] \simeq [X/G] \times_{\af{[X/G]}} [X/G].\]

Since by \cite{kawamata2004equivalences} the map $i_*$ is of Fourier-Mukai type, $h_*$ has both left and right adjoints, the log Hochschild cohomology for a firm orbifold as above can be defined similarly as in \ref{def:loghoch} by $h^! h_*$.

From now on, we follow \cite{arinkin2019formality} to describe the functors $h^*h_*$, $h^!h_*$ explicitly using twisted diagonals. 

First, we change the presentation of $[X/G]$ to the presentation $[X\times G/G\times G]$ where the action is given as $(g_1, g_2).(x, h) \mapsto (g_1.x, g_2hg_1^{-1})$. With this presentation, the diagonal map $[X/G]\to [X\times X/G\times G]$ becomes the quotient by $G\times G$ of the twisted diagonal maps
\[\Delta_g:X\to X\times X\]
defined as $x\mapsto (x,g.x)$. Since the diagonal map $[X/G]\to [X\times X/G\times G]$ factors through $[B/G\times G]$, we see that the twisted diagonal maps factors through $B$, obtaining maps $i_g:X\to B$ that we call the \textit{log twisted diagonal maps}. From the discussion above, we see that the log twisted diagonal maps can be explicitly described, namely $i_g(x)=(1,g).i(x)$.

Thus, we have
\[h^*h_*(-)\simeq \bigoplus_{g\in G}i_g^*i_*(-).\]
and
\[h^!h_*(-)\simeq \bigoplus_{g\in G}i_g^!i_*(-)\]
where the functors on the right-hand side are naturally equipped with an action of $G$. 


\subsection{Formality of orbifold Hochschild co/homology}

Let $C$ be a variety with closed subvarieties $A$ and $B$. We assume that both the embeddings $i:A\to C$ and $j:B\to C$ are local complete intersections. 
\begin{equation}\label{diag:lciofcarieties}
    \begin{tikzcd}
    A\cap B\ar[d, "q", swap] \ar[r, "p"] \pb & A\ar[d]\\
    B\ar[r] & C
\end{tikzcd}
\end{equation}
Similar to the case of self-intersections in Section \ref{sec:formselfint}, one can ask when the derived intersection $W:=A\times^h_C B$ is as simple as possible. 

\begin{definition}
    
The \emph{excess intersection bundle} $E$ is the cokernel of the natural map of normal bundles fitting into a short exact sequence
\[0\to N_{A\cap B/B}\to N_{A/C}|_{A\cap B}\to E\to 0.\]

\end{definition}

Analogously to section \ref{sec:loghoch}, when $E$ is locally free we have notions of formality:
\begin{itemize}
\item Algebraic formality is when the dg-algebra of the structure sheaf of $W$ is isomorphic to the dg-algebra of the symmetric algebra of the shifted dual excess bundle \eqref{eqn:formalalg}.
\item Geometric formality is when the derived intersection $W$ is the total space of the shifted excess bundle \eqref{eqn:formalgeom}. 
\end{itemize}

We generalize \cite[Proposition 3.3]{arinkin2019formality} to local complete intersections. The proof is straightforward. 

\begin{theorem}\label{thm:formalityderived}
    Consider the situation of Diagram \eqref{diag:lciofcarieties} and assume that the associated excess bundle $E$ is locally free. Assume the local complete intersection $A\to C$ possesses a quantized cycle and the short exact sequence of locally free sheaves
    \[0\to N_{A\cap B/B}\to N_{A/C}|_{A\cap B}\to E\to 0\]
    is split. 
    \begin{enumerate}
        \item The derived intersection is formal in the geometric sense, i.e, there exists an isomorphism of dg-schemes $W\simeq E[-1]$ over $A\times B$.
        \item The derived intersection is formal in the algebraic sense, i.e, there exists an isomorphism of dg-algebras $\OO_W\simeq \Sym(E^\vee[1])$ (once it is pushed forward to $A\times B$).
        \item There exists an isomorphism of dg endofunctors $\D(A)\to \D(B)$
    \[j^*i_*(-)\simeq q_*\left(p^*(-)\otimes \Sym(E^\vee[1])\right).\]
    \end{enumerate}
\end{theorem}

\begin{proof}

Omitted. 
    
\end{proof}


We use the theorem above in the case of the logarithmic twisted diagonals
\[i_g : X \to X \times X; \qquad x \mapsto (x, g.x)\]
of a firm action $G \action X$. These factor through $B$. 

\begin{definition}

The \textit{logarithmic fixed point locus} $X^g_{\log}$ is the underived intersection of $i:X\to B$ and $i_g:X\to B$ 
\begin{equation}\label{eqn:diagramfixedlocus}
    \begin{tikzcd}
    X^g_{\log} \ar[r, "p"] \ar[d, "q", swap] \lpbstrict & X\ar[d, "i"]\\
    X\ar[r, "i_g", swap] & B.
    \end{tikzcd}
\end{equation}

\end{definition}

Denote the excess bundle of this intersection by $E$. 



    


\begin{lemma}

    The schemes $X^g_{\log}$ are log smooth, and we have an isomorphism of bundles
    \[E^\vee\simeq \Omega^{1,\log}_{X^g_{\log}}=\left(\Omega^{1,\log}_X\right)_g\]
    where $\left(\Omega^{1,\log}_X\right)_g = \Omega^{1,\log}_X / \langle \omega - g \omega \rangle$ denotes the $g$-coinvariants of $\Omega^{1,\log}_X$.
    
\end{lemma}

\begin{proof}
    Consider the short exact sequence
    \[0\to E^\vee \to N^\vee_{X/B}|_{X^g_{\log}} \to N^\vee_{X^g_{\log}/X}\to 0\]
    arising from Diagram \eqref{eqn:diagramfixedlocus}. Identify the middle term with $\Omega^{1,\log}_X$. Since $i_g$ is obtained from $i$ via the action with a group element $g$ of finite order, we see that the last term is given by the $g$-invariants of $\Omega^{1,\log}_X$ by Cartan's lemma.

    Thus, the excess bundle can be identified with $\left(\Omega^{1,\log}_X\right)_g$ which is also a locally free sheaf. Using the embedding $q:X^g_{\log}\to X$, we see that $E^\vee\simeq \Omega^{1,\log}_{X^g_{\log}}=\left(\Omega^{1,\log}_X\right)_g$. Writing the distinguished triangle associated to the morphism $q$ one sees that the higher terms vanish, thus the schemes $X^g_{\log}$ are log smooth.

\end{proof}

Since the averaging map 
\[\left(\lkah{X}\right)_g\to\lkah{X}\]
defined as \[\omega\mapsto \sum_{i=1}^{o(g)}\frac{1}{o(g)} g^i\omega\]
splits the excess bundle short exact sequence, the conditions of Theorem \ref{thm:formalityderived} are satisfied.

\begin{corollary} The intersection depicted in Diagram \eqref{eqn:diagramfixedlocus} is formal in the sense of Theorem \ref{thm:formalityderived}. In particular, there are natural isomorphisms of endofunctors
    \[i_g^*i_*(-) \simeq q_*\left(p^*(-)\otimes \Sym\left(\left(\lkah{X}\right)_g[1]\right)\right)\]
    and
    \[i_g^!i_*(-) \simeq q_*\left(p^!(-) \otimes \Sym\left(\left(T^{\log}_X\right)^g[-1]\right)\right).\]\qed
\end{corollary}

In the realm of derived algebraic geometry, the corollary above says that the \textit{logarithmic derived loop space} of the orbifold $[X/G]$ defined as the derived fibre product $L^{\log}[X/G]:=[X/G]\times^h_{[B/G\times G]}[X/G]$ is isomorphic to the shifted tangent bundle on the \textit{logarithmic inertia stack} defined as the underived fibre product $I^{\log}[X/G]:=[X/G]\times_{[B/G\times G]}[X/G]$ as dg-schemes:
\[\mathbb{T}_{I^{\log}[X/G]}[-1]\simeq L^{\log}[X/G].\]


    

As a corollary, we obtain log version of the decomposition of orbifold Hochschild homology (see for instance, \cite{baranovsky2003orbifold, ganter2011inner} for the classical, non-logarithmic decomposition). 

\begin{corollary}\label{cor:logorbifolddecom}
We have isomorphisms of $k$-vector spaces
\[R^n\Gamma(\HHl{[X/G]}(\OO_{[X/G]}))=\left(\bigoplus_{g\in G} R^n\Gamma(\HHl{X^g_{\log}}(\OO_{X^g_{\log}}))\right)^G=\]
\[=\left(\bigoplus_{g\in G} \bigoplus_{q-p=n}H^p(X^g_{\log}, \Omega^{q,\log}_{X^g_{\log}})\right)^G.\]
\end{corollary}

\begin{proof}
    The proof is parallel to the proof of Corollary 1.17 in \cite{arinkin2019formality}.
\end{proof}

\begin{remark}
    We believe that the decomposition above can be generalized to other additive invariants in the sense of non-commutative motives in the sense of \cite{tabuada2018additive}. 
    
    Assume not only that the order $\# G$ is invertible in $k$, but that $k$ contains all $\# G$-roots of unity; let $k'$ be another field meeting these requirements. In \cite{tabuada2018additive}, Tabuada and van den Bergh show that for any \textit{additive invariant} $E:\dgcat_k\to D$ from the category of $k$-linear dg-categories to a $k'$-linear idempotent complete additive category, we have an isomorphism for any global quotient orbifold $[X/G]$ where $G$ acts on a smooth variety $X$,
    \[E([X/G])=\Big(\bigoplus_{g\in G}E(X^g)\Big)^G.\]
    In other words, the non-commutative motive of an orbifold decomposes with respect to the fixed point loci of the group action. 

    Even though there is no consensus on the definition of a derived category of coherent sheaves on a log scheme (Remark \ref{rem:nologcat}), one would expect that the non-commutative motive of a log orbifold decomposes with respect to the log fixed point loci as in Corollary \ref{cor:logorbifolddecom}. Therefore, we expect Corollary \ref{cor:logorbifolddecom} to hold for other additive invariants as well.
\end{remark}

We conclude the paper with a simple example in which there is no contribution to the Hochschild homology from the twisted sectors corresponding to the orbifold.

\begin{example}
Consider $\Aff^1$ with the negation action of $\mu_2 = \ZZ/2\ZZ$. We equip $\Aff^1$ with the toric log structure given by the Cartier divisor of the origin ($0$). Then $B \subseteq Bl_{\vec 0} \Aff^2$ is the complement of the strict transforms of the $x-$ and $y-$axes. The inclusions $i, i_{-1} : \Aff^1 \to B$ are the diagonal and anti-diagonal, which do not meet in the blowup.

Because the twisted diagonal $i_{-1}$ does not meet $i$, there is no contribution from $-1 \in \mu_2$. 
%
%
Thus, 
\[\HHl{[\Aff^1/\ZZ/2\ZZ]}(-) = h^* h_* (-) \simeq i^* i_* (-) \simeq (-)\otimes_{k[x]} k[x,t]\]
where $t$ has degree $-1$, and $\ZZ/2\ZZ$ acts as $x\mapsto -x$ and $t\mapsto t$.
\end{example}


\bibliographystyle{alpha}
\bibliography{refs.bib}

\newcommand{\etalchar}[1]{$^{#1}$}
\begin{thebibliography}{KKMSD73}

\bibitem[AC12a]{arinkin-caldararu}
Dima Arinkin and Andrei C{\u{a}}ld{\u{a}}raru.
\newblock When is the self-intersection of a subvariety a fibration?
\newblock {\em Advances in Mathematics}, 231(2):815--842, 2012.

\bibitem[AC12b]{arinkin2012self}
Dima Arinkin and Andrei C{\u{a}}ld{\u{a}}raru.
\newblock When is the self-intersection of a subvariety a fibration?
\newblock {\em Advances in Mathematics}, 231(2):815--842, 2012.

\bibitem[ACH19]{arinkin2019formality}
Dima Arinkin, Andrei C{\u{a}}ld{\u{a}}raru, and M{\'a}rton Hablicsek.
\newblock Formality of derived intersections and the orbifold {HKR}
  isomorphism.
\newblock {\em Journal of Algebra}, 540:100--120, 2019.

\bibitem[ACM{\etalchar{+}}15]{skeletonsfansabramchenmarcusulrischwise}
Dan {Abramovich}, Qile {Chen}, Steffen {Marcus}, Martin {Ulirsch}, and Jonathan
  {Wise}.
\newblock {Skeletons and fans of logarithmic structures}.
\newblock {\em arXiv e-prints}, page arXiv:1503.04343, March 2015.

\bibitem[ACMW17]{wisebounded}
Dan Abramovich, Qile Chen, Steffen Marcus, and Jonathan Wise.
\newblock Boundedness of the space of stable logarithmic maps.
\newblock {\em J. Eur. Math. Soc.}, 19:2783--2809, 2017.

\bibitem[AMW14]{abrammarcuswisecomparisontheoremsforgwisofsmoothpairs}
Dan Abramovich, Steffen Marcus, and Jonathan Wise.
\newblock Comparison theorems for {Gromov}-{Witten} invariants of smooth pairs
  and~of~degenerations.
\newblock {\em Annales de l'Institut Fourier}, 64(4):1611--1667, 2014.

\bibitem[AW18]{birationalinvarianceabramovichwise}
Dan Abramovich and Jonathan Wise.
\newblock Birational invariance in logarithmic {Gromov}-{Witten} theory.
\newblock {\em Compositio Mathematica}, 154(3):595–620, 2018.

\bibitem[Bar03]{baranovsky2003orbifold}
Vladimir Baranovsky.
\newblock Orbifold cohomology as periodic cyclic homology.
\newblock {\em International Journal of Mathematics}, 14(08):791--812, 2003.

\bibitem[BLP{\O}22]{HKRlogHH}
Federico {Binda}, Tommy {Lundemo}, Doosung {Park}, and Paul~Arne
  {{\O}stv{\ae}r}.
\newblock {A Hochschild-Kostant-Rosenberg theorem and residue sequences for
  logarithmic Hochschild homology}.
\newblock {\em arXiv e-prints}, page arXiv:2209.14182, September 2022.

\bibitem[BLR12]{neronmodels}
S.~Bosch, W.~L{\"u}tkebohmert, and M.~Raynaud.
\newblock {\em N{\'e}ron Models}.
\newblock Ergebnisse der Mathematik und ihrer Grenzgebiete. 3. Folge / A Series
  of Modern Surveys in Mathematics. Springer Berlin Heidelberg, 2012.

\bibitem[BM93]{bokstedt1993topological}
Marcel B{\"o}kstedt and Ib~Madsen.
\newblock {\em Topological cyclic homology of the integers}.
\newblock Aarhus Universitet. Matematisk Institut, 1993.

\bibitem[BZN12]{ben2012loop}
David Ben-Zvi and David Nadler.
\newblock Loop spaces and connections.
\newblock {\em Journal of Topology}, 5(2):377--430, 2012.

\bibitem[CHL23]{logquantumkproductchouherrlee}
You–Cheng Chou, Leo Herr, and Yuan–Pin Lee.
\newblock The log product formula in quantum {$K$}-theory.
\newblock {\em Mathematical Proceedings of the Cambridge Philosophical
  Society}, page 1–28, 2023.

\bibitem[CVdB10]{calaque2010hochschild}
Damien Calaque and Michel Van~den Bergh.
\newblock Hochschild cohomology and {Atiyah} classes.
\newblock {\em Advances in Mathematics}, 224(5):1839--1889, 2010.

\bibitem[EV85]{esnault1985logarithmic}
H{\'e}lene Esnault and Eckart Viehweg.
\newblock Logarithmic de {Rham} complexes and vanishing theorems.
\newblock 1985.

\bibitem[FG22]{ferrandgabberetaleenvelope}
Daniel Ferrand and Ofer Gabber.
\newblock {Enveloppe {\'e}tale de morphismes plats}.
\newblock {\em {Algebra \& Number Theory}}, 16(3), July 2022.

\bibitem[FM13]{loggeomintrohandbook}
G.~Farkas and I.~Morrison.
\newblock {\em Handbook of Moduli}.
\newblock Number v. 1 in Advanced lectures in mathematics. International Press,
  2013.

\bibitem[Gan11]{ganter2011inner}
Nora Ganter.
\newblock Inner products of 2-representations.
\newblock {\em arXiv preprint arXiv:1110.1711}, 2011.

\bibitem[Gri14]{grivaux2014hochschild}
Julien Grivaux.
\newblock The {Hochschild}--{Kostant}--{Rosenberg} isomorphism for quantized
  analytic cycles.
\newblock {\em International Mathematics Research Notices}, 2014(4):865--913,
  2014.

\bibitem[Gri20]{grivaux2020derived}
Julien Grivaux.
\newblock Derived geometry of the first formal neighbourhood of a smooth
  analytic cycle.
\newblock {\em Advances in Mathematics}, 361:106924, 2020.

\bibitem[Hab20]{hablicsek2020hodge}
M{\'a}rton Hablicsek.
\newblock Hodge theorem for the logarithmic de {Rham} complex via derived
  intersections.
\newblock {\em Research in the Mathematical Sciences}, 7:1--21, 2020.

\bibitem[Har77]{hartshorne}
Robin Hartshorne.
\newblock {\em Algebraic geometry}.
\newblock Springer-Verlag, New York, 1977.
\newblock Graduate Texts in Mathematics, No. 52.

\bibitem[{Her}23]{mythesislogprodfmla}
Leo {Herr}.
\newblock {The log product formula}.
\newblock {\em Algebra \& Number Theory}, 17, 2023.

\bibitem[HKR09]{kostant2009differential}
Gerhard Hochschild, Bertram Kostant, and Alex Rosenberg.
\newblock {\em Differential forms on regular affine algebras}.
\newblock Springer, 2009.

\bibitem[HMPW23]{rendimentodeicontiwisepandharipandeherrmymolcho}
Leo {Herr}, Sam {Molcho}, Rahul {Pandharipande}, and Jonathan {Wise}.
\newblock Rendimento dei conti.
\newblock {\em forthcoming}, 2023.

\bibitem[Hoc45]{hochschild1945cohomology}
Gerhard Hochschild.
\newblock On the cohomology groups of an associative algebra.
\newblock {\em Annals of Mathematics}, pages 58--67, 1945.

\bibitem[HWW23]{rhizomic}
Leo Herr, John Willis, and Jonathan Wise.
\newblock The rhizomic topology.
\newblock 2023.
\newblock forthcoming.

\bibitem[Kas87]{kassel1987cyclic}
Christian Kassel.
\newblock Cyclic homology, comodules, and mixed complexes.
\newblock {\em Journal of Algebra}, 107(1):195--216, 1987.

\bibitem[Kat89]{kato1989logarithmic}
Kazuya Kato.
\newblock Logarithmic structures of {Fontaine}-{Illusie}, {A}lgebraic analysis,
  geometry, and number theory (baltimore, md, 1988), 1989.

\bibitem[Kat99]{fkatomodulilogcurves}
Fumiharu Kato.
\newblock Log smooth deformation and moduli of log smooth curves.
\newblock 11, 11 1999.

\bibitem[Kaw04]{kawamata2004equivalences}
Yujiro Kawamata.
\newblock Equivalences of derived categories of sheaves on smooth stacks.
\newblock {\em American Journal of Mathematics}, 126(5):1057--1083, 2004.

\bibitem[Kel21]{keller2021hochschild}
Bernhard Keller.
\newblock Hochschild (co) homology and derived categories.
\newblock {\em Bulletin of the Iranian Mathematical Society}, 47(Suppl
  1):57--83, 2021.

\bibitem[KKMSD73]{kempfknudsenmumfordsaintdonatkkmstoroidalembeddings}
George Kempf, {Finn Faye} Knudsen, David Mumford, and Bernard Saint-Donat.
\newblock {\em Toroidal Embeddings 1}.
\newblock Lecture Notes in Mathematics. Springer Berlin Heidelberg, 1973.

\bibitem[KKN08]{logabelianvarieties2}
Takeshi Kajiwara, Kazuya Kato, and Chikara Nakayama.
\newblock Logarithmic abelian varieties.
\newblock {\em Nagoya Mathematical Journal}, 189:63–138, 2008.

\bibitem[Kon03]{kontsevich2003deformation}
Maxim Kontsevich.
\newblock Deformation quantization of {Poisson} manifolds.
\newblock {\em Letters in Mathematical Physics}, 66:157--216, 2003.

\bibitem[KS04]{katosaitologdiagonal}
Kazuya Kato and Takeshi Saito.
\newblock On the conductor formula of {Bloch}.
\newblock {\em Publications Math\'ematiques de l'IH\'ES}, 100:5--151, 2004.

\bibitem[LMB99]{laumonmoretbaillystacks}
G{\'e}rard Laumon and Laurent Moret-Bailly.
\newblock {\em Champs alg{\'e}briques}.
\newblock Ergebnisse der Mathematik und ihrer Grenzgebiete. 3. Folge / A Series
  of Modern Surveys in Mathematics. Springer Berlin Heidelberg, 1999.

\bibitem[Lod98]{loday1998cyclic}
Jean-Louis Loday.
\newblock Cyclic homology. {A}ppendix {E} by {M}ar{\i}a {O}. {R}onco. chapter
  13 by the author in collaboration with {T}eimuraz {P}irashvili.
\newblock {\em Grundlehren der Mathematischen Wissenschaften [Fundamental
  Principles of Mathematical Sciences]}, 301, 1998.

\bibitem[LVdB05]{lowen-vandenbergh}
Wendy Lowen and Michel Van~den Bergh.
\newblock Hochschild cohomology of abelian categories and ringed spaces.
\newblock {\em Advances in Mathematics}, 198(1):172--206, 2005.

\bibitem[MS80]{mehta1980moduli}
Vikram~Bhagvandas Mehta and Conjeevaram~Srirangachari Seshadri.
\newblock Moduli of vector bundles on curves with parabolic structures.
\newblock {\em Mathematische annalen}, 248(3):205--239, 1980.

\bibitem[MW18]{logpic}
Samouil {Molcho} and Jonathan {Wise}.
\newblock {The logarithmic Picard group and its tropicalization}.
\newblock {\em arXiv e-prints}, page arXiv:1807.11364, July 2018.

\bibitem[Nee96]{neeman1996grothendieck}
Amnon Neeman.
\newblock The {G}rothendieck duality theorem via {B}ousfield’s techniques and
  {B}rown representability.
\newblock {\em Journal of the American Mathematical Society}, 9(1):205--236,
  1996.

\bibitem[Ogu18]{ogusloggeom}
Arthur Ogus.
\newblock {\em Lectures on Logarithmic Algebraic Geometry}.
\newblock Cambridge Studies in Advanced Mathematics. Cambridge University
  Press, 2018.

\bibitem[Ols03]{logstacks}
Martin Olsson.
\newblock Logarithmic geometry and algebraic stacks.
\newblock {\em Scientific Annals of the Ecole Normale Sup\'erieure}, 36, 2003.

\bibitem[Ols05]{logcotangent}
Martin Olsson.
\newblock The logarithmic cotangent complex.
\newblock {\em Mathematische Annalen}, 2005.

\bibitem[Ols24]{olssondraftloghochschild}
Martin Olsson.
\newblock Hochschild homology for log schemes.
\newblock 2024.
\newblock Forthcoming. See accompanying video
  \url{https://www.youtube.com/watch?v=vPkSZm8DOYk} for an overview.

\bibitem[Rog09]{rognesoriginalloghh}
John Rognes.
\newblock Topological logarithmic structures.
\newblock {\em Geometry and Topology Monographs}, pages 401--544, 2009.

\bibitem[Rom09]{romagnypi0}
Matthieu Romagny.
\newblock Composantes connexes et irréductibles en familles.
\newblock {\em Manuscripta Mathematica}, 136, 12 2009.

\bibitem[RSS14]{rognes2014}
John Rognes, Steffen Sagave, and Christian Schlichtkrull.
\newblock Logarithmic topological {H}ochschild homology of topological
  {K}-theory spectra.
\newblock {\em Journal of the European Mathematical Society}, 20, 10 2014.

\bibitem[RSS15]{rognes2015}
John Rognes, Steffen Sagave, and Christian Schlichtkrull.
\newblock Localization sequences for logarithmic topological {H}ochschild
  homology.
\newblock {\em Mathematische Annalen}, 363(3):1349--1398, Dec 2015.

\bibitem[SST20]{scherotzke2020parabolic}
Sarah Scherotzke, Nicolo Sibilla, and Mattia Talpo.
\newblock Parabolic semi-orthogonal decompositions and kummer flat invariants
  of log schemes.
\newblock {\em Documenta Mathematica}, 25:955--1009, 2020.

\bibitem[{Sta}23]{sta}
The {Stacks Project Authors}.
\newblock \textit{Stacks Project}.
\newblock \url{http://stacks.math.columbia.edu}, 2023.

\bibitem[Swa96]{swan1996hochschild}
Richard~G Swan.
\newblock Hochschild cohomology of quasiprojective schemes.
\newblock {\em Journal of Pure and Applied Algebra}, 110(1):57--80, 1996.

\bibitem[To{\"e}12]{toen2012proper}
Bertrand To{\"e}n.
\newblock Proper local complete intersection morphisms preserve perfect
  complexes.
\newblock {\em arXiv preprint arXiv:1210.2827}, 2012.

\bibitem[TV07]{toen2007note}
Bertrand To{\"e}n and Gabriele Vezzosi.
\newblock A note on {Chern} character, loop spaces and derived algebraic
  geometry. available at.
\newblock {\em arXiv preprint arXiv:0804.1274}, 2007.

\bibitem[TV18]{talpo2018infinite}
Mattia Talpo and Angelo Vistoli.
\newblock Infinite root stacks and quasi-coherent sheaves on logarithmic
  schemes.
\newblock {\em Proceedings of the London Mathematical Society},
  116(5):1187--1243, 2018.

\bibitem[TVdB18]{tabuada2018additive}
Gon{\c{c}}alo Tabuada and Michel Van~den Bergh.
\newblock Additive invariants of orbifolds.
\newblock {\em Geometry \& Topology}, 22(5):3003--3048, 2018.

\bibitem[Uli13]{functorialtropicalizationulirsch}
Martin Ulirsch.
\newblock Functorial tropicalization of logarithmic schemes: The case of
  constant coefficients.
\newblock {\em Proceedings of the London Mathematical Society}, 114, 10 2013.

\bibitem[Vai17]{vaintrob2017categorical}
Dmitry Vaintrob.
\newblock Categorical {L}ogarithmic {Hodge} {T}heory, {I}.
\newblock {\em arXiv preprint arXiv:1712.00045}, 2017.

\bibitem[Yok95]{yokogawa1995infinitesimal}
K{\^o}ji Yokogawa.
\newblock Infinitesimal deformation of parabolic {H}iggs sheaves.
\newblock {\em International Journal of Mathematics}, 6(1):125, 1995.

\bibitem[Yu16]{yu2016dolbeault}
Shilin Yu.
\newblock The {D}olbeault dga of a formal neighborhood.
\newblock {\em Transactions of the American Mathematical Society},
  368(11):7809--7843, 2016.

\bibitem[Yu19]{yu2019todd}
Shilin Yu.
\newblock Todd class via homotopy perturbation theory.
\newblock {\em Advances in Mathematics}, 352:297--325, 2019.

\end{thebibliography}

\end{document}